\documentclass[10pt,oneside]{amsart}

\usepackage[
paper=a4paper,
headsep=15pt,text={135mm,216mm},centering
]{geometry}

\usepackage[bookmarks]{hyperref}
\usepackage{amssymb,amsxtra}
\usepackage{graphicx,xcolor,mathrsfs}
\usepackage{float,array,calc,booktabs,stackrel,url}
\usepackage{pb-diagram} % Diagram drawing
\usepackage{enumitem}\setlist[enumerate,1]{font=\upshape}\setlist[enumerate,2]{font=\upshape}
\usepackage{mathtools}
\usepackage{tikz}
\usetikzlibrary{
  cd,
  calc,
  positioning,
  arrows,
  decorations.pathreplacing,
  decorations.markings,
}
\usepackage[mode=buildnew]{standalone}

\definecolor{todo-background-color}{gray}{0.95}
\usepackage[
  textwidth=1.1in,
  backgroundcolor=todo-background-color,
  bordercolor=black,
  linecolor=black,
  textsize=footnotesize
]{todonotes}

\iftrue
    
    \makeatletter
    \def\@settitle{%
      \vspace*{-10pt}
      \begin{flushleft}%
        \LARGE\bfseries
        \strut\@title\strut
      \end{flushleft}%
    }
    \def\@setauthors{%
      \begingroup
      \def\thanks{\protect\thanks@warning}%
      \trivlist
      \raggedright
      \large \@topsep27\p@\relax
      \advance\@topsep by -\baselineskip
    \item\relax
      \author@andify\authors
      \def\\{\protect\linebreak}%
      \authors
      \ifx\@empty\contribs
      \else
      ,\penalty-3 \space \@setcontribs
      \@closetoccontribs
      \fi
      \normalfont
      \endtrivlist
      \endgroup
    }
    \def\@setaddresses{\par
      \nobreak \begingroup
      \small\raggedright
      \def\author##1{\nobreak\addvspace\smallskipamount}%
      \def\\{\unskip, \ignorespaces}%
      \interlinepenalty\@M
      \def\address##1##2{\begingroup
        \par\addvspace\bigskipamount\noindent
        \@ifnotempty{##1}{(\ignorespaces##1\unskip) }%
        {\ignorespaces##2}\par\endgroup}%
      \def\curraddr##1##2{\begingroup
        \@ifnotempty{##2}{\nobreak\noindent\curraddrname
          \@ifnotempty{##1}{, \ignorespaces##1\unskip}\/:\space
          ##2\par}\endgroup}%
      \def\email##1##2{\begingroup
        \@ifnotempty{##2}{\nobreak\noindent E-mail address%
          \@ifnotempty{##1}{, \ignorespaces##1\unskip}\/:\space
          \ttfamily##2\par}\endgroup}%
      \def\urladdr##1##2{\begingroup
        \def~{\char`\~}%
        \@ifnotempty{##2}{\nobreak\noindent\urladdrname
          \@ifnotempty{##1}{, \ignorespaces##1\unskip}\/:\space
          \ttfamily##2\par}\endgroup}%
      \addresses
      \endgroup
      \global\let\addresses=\@empty
    }
    \def\@setabstracta{%
      \ifvoid\abstractbox
      \else
      \skip@17pt \advance\skip@-\lastskip
      \advance\skip@-\baselineskip \vskip\skip@
      \box\abstractbox
      \prevdepth\z@ % because \abstractbox is a vtop
      \vskip-28pt
      \fi
    }
    \renewenvironment{abstract}{%
      \ifx\maketitle\relax
      \ClassWarning{\@classname}{Abstract should precede
        \protect\maketitle\space in AMS document classes; reported}%
      \fi
      \global\setbox\abstractbox=\vtop \bgroup
      \normalfont\small
      \list{}{\labelwidth\z@
        \leftmargin0pc \rightmargin\leftmargin
        \listparindent\normalparindent \itemindent\z@
        \parsep\z@ \@plus\p@
        
      }%
    \item[\hskip\labelsep\bfseries\abstractname.]%
    }{%
      \endlist\egroup
      \ifx\@setabstract\relax \@setabstracta \fi
    }

    \def\ps@headings{\ps@empty
      \def\@evenhead{%
        \setTrue{runhead}%
        \normalfont\scriptsize
        \rlap{\thepage}\hfill
        \def\thanks{\protect\thanks@warning}%
        \leftmark{}{}}%
      \def\@oddhead{%
        \setTrue{runhead}%
        \normalfont\scriptsize
        \def\thanks{\protect\thanks@warning}%
        \rightmark{}{}\hfill \llap{\thepage}}%
      \let\@mkboth\markboth
    }\ps@headings

    \def\section{\@startsection{section}{1}%
      \z@{-1.4\linespacing\@plus-.5\linespacing}{.8\linespacing}%
      {\normalfont\bfseries\Large}}
    \def\subsection{\@startsection{subsection}{2}%
      \z@{-.8\linespacing\@plus-.3\linespacing}{.5\linespacing\@plus.2\linespacing}%
      {\normalfont\bfseries\large}}
    \def\subsubsection{\@startsection{subsubsection}{3}%
      \z@{.7\linespacing\@plus.2\linespacing}{-1.5ex}%
      {\normalfont\itshape}}
    \def\paragraph{\@startsection{paragraph}{4}%
      \z@{.7\linespacing\@plus.2\linespacing}{-1.5ex}%
      {\normalfont\itshape}}
    \def\@secnumfont{\bfseries}

    \renewcommand\contentsnamefont{\bfseries}
    \def\@starttoc#1#2{\begingroup
      \setTrue{#1}%
      \par\removelastskip\vskip\z@skip
      \@startsection{}\@M\z@{\linespacing\@plus\linespacing}%
      {.5\linespacing}{%\centering
        \contentsnamefont}{#2}%
      \ifx\contentsname#2%
      \else \addcontentsline{toc}{section}{#2}\fi
      \makeatletter
      \@input{\jobname.#1}%
      \if@filesw
      \@xp\newwrite\csname tf@#1\endcsname
      \immediate\@xp\openout\csname tf@#1\endcsname \jobname.#1\relax
      \fi
      \global\@nobreakfalse \endgroup
      \addvspace{32\p@\@plus14\p@}%
      \let\tableofcontents\rela\x
    }
    \def\contentsname{Contents}
    \def\l@section{\@tocline{2}{.5ex}{0mm}{5pc}{}}
    \def\l@subsection{\@tocline{2}{0pt}{2em}{5pc}{}}
    \makeatother

\fi
\def\to{\mathchoice{\longrightarrow}{\rightarrow}{\rightarrow}{\rightarrow}}
\makeatletter
\newcommand{\shortxra}[2][]{\ext@arrow 0359\rightarrowfill@{#1}{#2}}
\def\longrightarrowfill@{\arrowfill@\relbar\relbar\longrightarrow}
\newcommand{\longxra}[2][]{\ext@arrow 0359\longrightarrowfill@{#1}{#2}}
\renewcommand{\xrightarrow}[2][]{\mathchoice{\longxra[#1]{#2}}%
  {\shortxra[#1]{#2}}{\shortxra[#1]{#2}}{\shortxra[#1]{#2}}}
\makeatother

\makeatletter
\def\addtagsub#1{\let\oldtf=\tagform@\def\tagform@##1{\oldtf{##1}\hbox{$_{#1}$}}}
\makeatother

\makeatletter
\def\Nopagebreak{\@nobreaktrue\nopagebreak}
\makeatother

\makeatletter
\newtheoremstyle{theorem-giventitle}
        {}{}              %%% space between body and thm
        {\itshape}                      %%% Thm body font
        {}                              %%% Indent amount (empty = no indent)
        {\bfseries}                     %%% Thm head font
        {.}                             %%% Punctuation after thm head
        {\thm@headsep}                             %%% Space after thm head
        {\thmnote{\bfseries#3}}%%% Thm head spec
\newtheoremstyle{theorem-givenlabel}
        {}{}              %%% space between body and thm
        {\itshape}                      %%% Thm body font
        {}                              %%% Indent amount (empty = no indent)
        {\bfseries}                     %%% Thm head font
        {.}                             %%% Punctuation after thm head
        {\thm@headsep}                             %%% Space after thm head
        {\thmname{#1}~\thmnumber{#3}\setcurrentlabel{#3}}%%% Thm head spec
\newtheoremstyle{definition-giventitle}
        {}{}              %%% space between body and thm
        {}                      %%% Thm body font
        {}                              %%% Indent amount (empty = no indent)
        {\bfseries}                     %%% Thm head font
        {.}                             %%% Punctuation after thm head
        {\thm@headsep}                             %%% Space after thm head
        {\thmnote{\bfseries#3}}%%% Thm head spec
\def\setcurrentlabel#1{\gdef\@currentlabel{#1}}
\makeatother

\newtheorem{theorem}{Theorem}[section]

\newtheorem{proposition}[theorem]{Proposition}

\newtheorem{lemma}[theorem]{Lemma}

\theoremstyle{definition}
\newtheorem{definition}[theorem]{Definition}

\newtheorem{remark}[theorem]{Remark}

\newtheorem*{case2'}{Case 2$'$}

\theoremstyle{theorem-giventitle}
\newtheorem{theorem-named}{}
\theoremstyle{theorem-givenlabel}
\newtheorem{theorem-labeled}{Theorem}

\theoremstyle{definition-giventitle}
\newtheorem{definition-named}{}
\newtheorem{conjecture-named}{}
\newtheorem{case-named}{}

\numberwithin{equation}{section}

\def\d{\partial}
\def\Z{\mathbb{Z}}

\def\Q{\mathbb{Q}}
\def\C{\mathbb{C}}

\def\cP{\mathcal{P}}

\def\cT{\mathcal{T}}
\def\scT{\mathscr{T}}
\def\scS{\mathscr{S}}

\def\hat{\widehat}
\def\sm{\smallsetminus}
\DeclareMathOperator\Ker{Ker}

\def\Im{\operatorname{Im}}

\DeclareMathOperator\sign{sign}

\DeclareMathOperator\Spin{Spin}
\def\Spinc{\Spin^c}

\def\osigmat{\bar\sigma^{(2)}}
\def\cC{\mathcal{C}}

\makeatletter

\begin{document}

\title[One-bipolar topologically slice knots and primary decomposition]{One-bipolar topologically slice knots and primary decomposition}

\author{Min Hoon Kim}
\address{
  Department of Mathematics\\
  POSTECH \\
  Pohang Gyeongbuk 37673\\
  Republic of Korea
}
\email{kminhoon@gmail.com}

\author{Se-Goo Kim}
\address{Department of Mathematics and Research Institute for Basic Sciences\\
Kyung Hee University\\
Seoul 02447\\
Republic of Korea
}
\email{sgkim@khu.ac.kr}

\author{Taehee Kim}
\address{Department of Mathematics\\
  Konkuk University\\
  Seoul 05029\\
  Republic of Korea
}
\email{tkim@konkuk.ac.kr}

\thanks{The first named author was partly supported by the POSCO TJ Park
Science Fellowship. The second named author was supported by the Basic Science Research Program through the National Research Foundation of Korea (NRF) funded by the Ministry of Education (NRF-2015R1D1A1A01058384). The last named author was supported by Basic Science Research Program through the National Research Foundation of Korea (NRF) funded by the Ministry of Education (no.2018R1D1A1B07048361).}

\def\subjclassname{\textup{2010} Mathematics Subject Classification}
\expandafter\let\csname subjclassname@1991\endcsname=\subjclassname
\expandafter\let\csname subjclassname@2000\endcsname=\subjclassname
%\keywords{}

\begin{abstract}
Let $\{\cT_n\}$ be the bipolar filtration of the smooth concordance group of topologically slice knots, which was introduced by Cochran, Harvey, and Horn. It is known that for each $n\ne 1$ the group $\cT_n/\cT_{n+1}$ has infinite rank and $\cT_1/\cT_2$ has positive rank. In this paper, we show that $\cT_1/\cT_2$ also has infinite rank. Moreover, we prove that there exist infinitely many Alexander polynomials $p(t)$ such that there exist infinitely many knots in $\cT_1$ with Alexander polynomial $p(t)$ whose nontrivial linear combinations are not concordant to any knot with Alexander polynomial coprime to $p(t)$, even modulo $\cT_2$. This extends the recent result of Cha on the primary decomposition of $\cT_n/\cT_{n+1}$ for all $n\ge 2$ to the case $n=1$. 

To prove our theorem, we show that the surgery manifolds of satellite links of $\nu^+$-equivalent knots with the same pattern link have the same Ozsv\'{a}th-Szab\'{o} $d$-invariants, which is of independent interest. As another application, for each $g\ge 1$, we give a topologically slice knot of concordance genus $g$ that is $\nu^+$-equivalent to the unknot.
\end{abstract}

\maketitle

%\tableofcontents
\section{Introduction}
Since Freedman \cite{Freedman:1982-1, Freedman:1984-1} and Donaldson \cite{Donaldson:1983-1} revealed the difference between the smooth and topological structures on 4-manifolds, classifying the structure of the concordance group of topologically slice knots, which we denote by $\cT$, has been one of the central research subjects. Freedman showed that knots with trivial Alexander polynomial are topologically slice \cite{Freedman:1982-2, Freedman-Quinn:1990-1}, so $\cT$ has a subgroup $\Delta$ generated by knots with trivial Alexander polynomial. By using gauge theory and Heegaard Floer homology, it is known that $\cT$ and $\Delta$ have $\Z^\infty$-summands \cite{Endo:1995-1, Hedden-Kirk:2011-1, Hedden-Kirk:2010-1,Ozsvath-Stipsicz-Szabo:2014-1,DHST:2019-1,Kim-Park:2016-1}. The difference between $\cT$ and $\Delta$ was found recently, and it is known that the group $\cT/\Delta$ has a $\Z^\infty\oplus \Z_2^\infty$ subgroup \cite{Hedden-Livingston-Ruberman:2010-01,Hedden-Kim-Livingston:2016-1,Cochran-Horn:2012-1}. 

Cochran, Harvey, and Horn \cite{Cochran-Harvey-Horn:2012-1} established a filtration indexed by nonnegative integers
\[
0\subset \cdots \subset \cT_{n+1}\subset \cT_n \subset \cdots \subset \cT_0 \subset\cT
\]
which is called the \emph{bipolar filtration} of $\cT$, where $\cT_n$ is the subgroup of (the concordance classes of) \emph{$n$-bipolar} knots.  We give a definition of $\cT_n$ in Definition~\ref{definition:n-bipolar}. Briefly, the notion of an $n$-bipolar knot is based on the idea of relating Donaldson's diagonalization theorem to the covering spaces of a 4-manifold with boundary the zero-framed surgery on the knot in $S^3$ corresponding to the derived series of the fundamental group of the 4-manifold. The bipolar filtration of $\cT$ also descends to a filtration of $\scT:=\cT/\Delta$:
\[
0\subset \cdots \subset \scT_{n+1}\subset \scT_n \subset \cdots\subset \scT_1 \subset \scT_0\subset \scT
\] 
where $\scT_{n}:=\cT_n/(\cT_n\cap \Delta)$.

Many concordance invariants derived from the Floer homology theory by Ozsv\'{a}th and Szab\'{o} vanish on $\cT_0$ or $\cT_1$. In \cite{Cochran-Harvey-Horn:2012-1}, it was shown that the $\tau$-invariant \cite{Ozsvath-Szabo:2003-1},  $\varepsilon$-invariant \cite{Hom:2011-1}, and $\delta_p$-invariant \cite{Manolescu-Owens:2007-1,Jabuka:2012-1} vanish for knots in $\cT_0$. It is known that the $\nu^+$-invariant \cite{Hom-Wu:2016-1}, the $\Upsilon$-invariant \cite{Ozsvath-Stipsicz-Szabo:2014-1}, and the $\varphi_j$-invariants \cite{DHST:2019-1} vanish for knots in $\cT_0$, and also the various invariants in \cite{Jabuka-Naik:2007-1,Grigsby-Ruberman-Strle:2008-1,Greene-Jabuka:2011-1} derived from the correction term $d$-invariants of Ozsv\'{a}th and Szab\'{o} \cite{Ozsvath-Szabo:2003-2} vanish for knots in $\cT_1$ (see \cite[Section~6]{Cochran-Harvey-Horn:2012-1}).

Nevertheless, combining the $d$-invariant with topological concordance invariants such as the von Neumann $\rho$-invariant and the Casson-Gordon invariant, it was found that the bipolar filtrations are highly nontrivial; it was shown that $\cT/\cT_0$ has infinite rank by Cochran, Harvey, and Horn \cite{Cochran-Harvey-Horn:2012-1}. It was also shown that $\cT_n/\cT_{n+1}$ has infinite rank for $n=0$ by Cochran and Horn \cite{Cochran-Horn:2012-1} and for $n\ge 2$ by Cha and the first named author \cite{Cha-Kim:2017-1}. 
For the filtration $\{\scT_n\}$ of $\scT$, Cochran and Horn \cite{Cochran-Horn:2012-1} showed that $\scT_0/\scT_1$ has infinite rank, and Cha \cite{Cha:2019-1} showed that $\scT_n/\scT_{n+1}$ has infinite rank for each $n\ge 2$.

But for the remaining case $n=1$, it is only known that $\cT_1/\cT_2$ has positive rank \cite{Cochran-Harvey-Horn:2012-1} and it remains as an open question whether or not $\cT_1/\cT_2$ and $\scT_1/\scT_2$ have infinite rank. The first main result of this paper is to answer the question:

\begin{theorem}\label{theorem:T_1/T_2}
	The groups $\cT_1/\cT_2$ and $\scT_1/\scT_2$ have infinite rank.
\end{theorem}

Very recently, Cha \cite{Cha:2019-1} presented a framework for the study on the primary decomposition of $\cT$. Theorem~\ref{theorem:T_1/T_2} immediately follows from Theorem~\ref{theorem:main} below which reveals a new structure of the primary decomposition of $\cT_1/\cT_2$ and $\scT_1/\scT_2$ by extending Theorems~C and D in \cite{Cha:2019-1} for $n=1$.

To state our result more explicitly, we set up notations. 

\begin{figure}[htb!]
\includegraphics{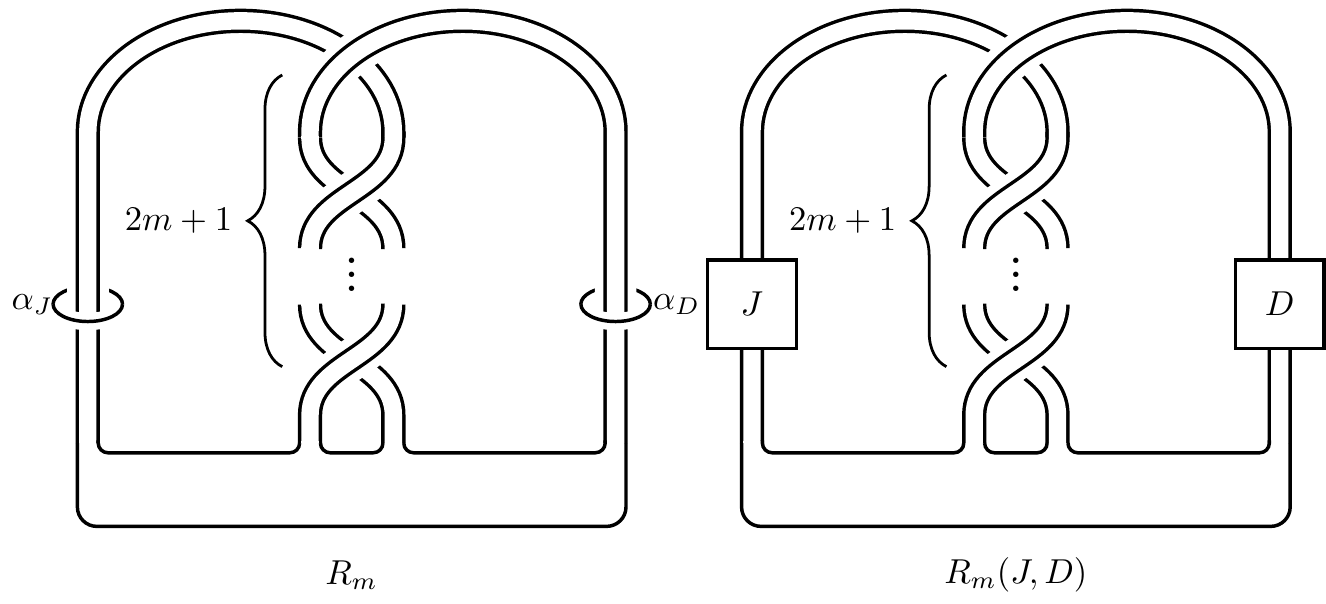}
\caption{The knot $R_m$ and the satellite knot $R_m(J,D)$.}\label{figure:seed-knot}
\end{figure}
Let $R_m$ be the genus one knot with Alexander polynomial $\lambda_m=(mt-(m+1))((m+1)t-m)\in \Z[t^{\pm 1}]$ and $\alpha_J$ and $\alpha_D$ be the curves dual to the bands of the Seifert surface of $R_m$ as depicted in Figure~\ref{figure:seed-knot}. For knots $J$ and $D$, we denote by $R_m(J, D)$ the knot obtained via the satellite construction which ties the knot $J$ (resp.\ $D$) through the band dual to the curve $\alpha_J$ (resp.\ $\alpha_D$) as depicted in Figure~\ref{figure:seed-knot}. For a knot $K$ and $t\in \Z$, let $D_+(K,t)$ (resp.\ $D_-(K,t)$) denote the $t$-twisted positive-clasped (resp.\ negative-clasped) Whitehead double of $K$. Let $D:=D_+(T_{2,3}, 0)$ for the right-handed trefoil $T_{2,3}$.  For $i\ge 1$, let $p_i$ be primes such that $5= p_1 < p_2 <\cdots$ and each $p_i$ is congruent to 1 modulo 4. Let $k_i:=\frac18(p_i-1)^2$ and $J_i:=D_-(k_i T_{2,3}, 2k_i)$. 

For a knot $K$ let $\Delta_K$ denote the Alexander polynomial of $K$. Let
\[
\cT^{\lambda_m}:=\{[K]\in \cC\,\mid\, \Delta_K \mbox{ is coprime to }\lambda_m\},
\]
and for a subgroup $S$ of $\cC$ let $S^{\lambda_m}:=S\cap \cT^{\lambda_m}$, $\scS :=S/S\cap \Delta$, and $\scS^{\lambda_m}:=S^{\lambda_m}/S^{\lambda_m}\cap \Delta$. Note that if $S$ is a subgroup of $\cT$, $\scS$ is a subgroup of $\scT$. The groups $\cT_1/(\cT_1^{\lambda_m} +\cT_2)$ and $\scT_1/(\scT_1^{\lambda_m} +\scT_2)$ are isomorphic since $\cT_1\cap \Delta$ is contained in $\cT_1^{\lambda_m}$.
\begin{theorem}\label{theorem:main} Let $m$ be an odd positive integer. For all $i\ge 1$, the knots $R_m(J_i,D)$, which have Alexander polynomial $\lambda_m$, are in $\cT_1$. If $K$ is a nontrivial linear combination of $R_m(J_i, D)$, then $K$ is not concordant to any knot with Alexander polynomial coprime to $\lambda_m$ modulo $\cT_2$. In particular, the groups $\cT_1/(\cT_1^{\lambda_m} +\cT_2)$ and $\scT_1/(\scT_1^{\lambda_m} +\scT_2)$ have infinite rank.   
\end{theorem}

Since $\cT_1/\cT_2$ and $\scT_1/\scT_2$ surject to $\cT_1/(\cT_1^{\lambda_m} +\cT_2)$ and $\scT_1/(\scT_1^{\lambda_m} +\scT_2)$, respectively, Theorem~\ref{theorem:T_1/T_2} follows immediately from Theorem~\ref{theorem:main}. We remark that Theorem~\ref{theorem:main} extends \cite[Theorems~C and D]{Cha:2019-1} which show that the groups $\cT_n/(\cT_n^{\lambda_m} + \cT_{n+1})$ and $\scT_n/(\scT_n^{\lambda_m} + \scT_{n+1})$ have infinite rank for $n\geq 2$. 

To prove Theorem~\ref{theorem:main} we use the arguments in \cite{Cochran-Harvey-Horn:2012-1,Cha-Kim:2017-1,Cha:2019-1}, which combine the $d$-invariant and the von Neumann $\rho$-invariant, and in particular estimate the $d$-invariants of branched covers of $S^3$ along $R_m(J,D)$ for some knots $J$ (see \cite[Lemma~8.2]{Cochran-Harvey-Horn:2012-1} and \cite[Lemma~5.1]{Cha-Kim:2017-1}). To be precise, in \cite[Lemma~8.2]{Cochran-Harvey-Horn:2012-1} and \cite[Lemma~5.1]{Cha-Kim:2017-1}, the authors show that certain $d$-invariants of branched covers of $S^3$ along $R_m(U,D)$ are negative where $U$ is the unknot, and the knots $J$ are needed to be unknotted by changing only positive crossings to conclude that the corresponding $d$-invariants of branched covers of $S^3$ along $R_m(J,D)$ are also negative. Due to this restriction on $J$, one could not apply directly the arguments in \cite{Cochran-Harvey-Horn:2012-1,Cha-Kim:2017-1,Cha:2019-1} to the case of $\cT_1/\cT_2$ and show $\cT_1/\cT_2$ has infinite rank; note that our knots $J_i$ in Theorem~\ref{theorem:main} are not unknotted by changing only positive crossings. 

To resolve this issue and prove Theorem~\ref{theorem:main}, we show that the surgery manifolds of satellite links of $\nu^+$-equivalent knots with the same pattern link have the same $d$-invariants, which is of independent interest. (For the definition of $\nu^+$-equivalence, see Definition~\ref{definiton:nu+equivalence}.) In the following, $P(K)$ and $P(J)$ denote the satellite links of knots $K$ and $J$ with pattern link $P$, respectively. 
\begin{theorem}\label{theorem:d-invariant-satellite-intro}Let $K$ and $J$ be knots and $P$ be an ordered $k$-component link in $S^1\times D^2$. Let $\Lambda=(n_1,\ldots,n_k)$ be surgery coefficients such that $S^3_{\Lambda}(P(K))$ and $S^3_{\Lambda}(P(J))$ are $\Z_2$-homology spheres. If $K$ and $J$ are $\nu^+$-equivalent, then 
\[d(S^3_{\Lambda}(P(K)),\mathfrak{t}_K)=d(S^3_{\Lambda}(P(J)),\mathfrak{t}_J)\]
 where $\mathfrak{t}_K$ and $\mathfrak{t}_J$ are $\operatorname{Spin}^{c}$ structures satisfying $c_1(\mathfrak{t}_K)=c_1(\mathfrak{t}_J)$ under the bijection $H^2(S^3_{\Lambda}(P(K)))\to H^2(S^3_{\Lambda}(P(J)))$.
\end{theorem}

Theorem~\ref{theorem:d-invariant-satellite-intro} is proved in Subsection~\ref{subsection:satellite-links}. This can be regarded as a link analogue of the fact that satellite knots of $\nu^+$-equivalent knots with the same pattern knot are $\nu^+$-equivalent which was proved in \cite{Kim-Park:2016-1,Sato:2017-1}.

Theorem~\ref{theorem:main} has an application to concordance genus. Recall that the \emph{concordance genus} of a knot is the minimum genus among all knots concordant to the knot. Using the $\varepsilon$-invariant Hom \cite{Hom:2015-1}  showed that for each $g\ge 1$ there exists a topologically slice knot of concordance genus $g$ and 4-genus one. There are also lower bounds on concordance genus obtained from the $\Upsilon$-invariant \cite{Ozsvath-Stipsicz-Szabo:2014-1} and the $\varphi_j$-invariants \cite{DHST:2019-1}. Note that if a knot is $\nu^+$-equivalent to the unknot, then it has vanishing $\varepsilon$-invariant and $\Upsilon$-invariant \cite{Hom:2017-1}, and it also has vanishing $\varphi_j$-invariants \cite{DHST:2019-1}. 

\begin{theorem}\label{theorem:concordance-genus}
For each $g\ge 1$, there exists a topologically slice knot of concordance genus $g$ which is $\nu^+$-equivalent to the unknot.
\end{theorem}
\begin{proof}
For $1\le i\le g$, let $K_i:=R_{2i-1}(J_1,D)$ and $K:=\#_{i=1}^gK_i$ (recall that $J_1=D_-(2T_{2,3}, 4)$ and $D=D_+(T_{2,3},0)$). Then, $K$ is a topologically slice knot of genus $g$ and the degree of $\Delta_K$ is $2g$. 
 
 Since $K$ is in $\cT_1$, by \cite[Corollary~6.7]{Cochran-Harvey-Horn:2012-1} $d(S^3_1(K))=0$. Therefore $V_0(K)=-\frac12d(S^3_1(K))=0$ where $V_0$ is a knot concordance invariant given in Subsection~2.2 of \cite{Ni-Wu:2015-1}. Since $\nu^+$ is the smallest $k\ge 0$ such that $V_k=0$, it follows that $\nu^+(K)=0$ and hence $K$ is $\nu^+$-equivalent to the unknot. 
 
 Suppose $J$ is a knot of genus $g'<g$ and $K\# -J$ is slice. Since the degree of $\Delta_J$ is less than $2g$, there is some $j\in \{1,2,\ldots, g\}$ such that $\Delta_{K_j}$ is not a factor of $\Delta_J$. Then, for $K':=((\#_{i\ne j}K_i)\# -J)$, we have $K_j\# K' = K\# -J$ is slice and $\Delta_{K_j}$ is coprime to $\Delta_{K'}$. This contradicts Theorem~\ref{theorem:main}. 
\end{proof}

We remark that Theorem~\ref{theorem:concordance-genus} can also be shown using the knots in \cite[Theorem~D]{Cha:2019-1}.

If a topologically slice knot $K$ is not in $\cT_1$, one cannot show that $K$ is nontrivial in $\cT/\cT^{\Delta_K}$ using the arguments in the proof of Theorem~\ref{theorem:main}. In Appendix, combining the results in \cite{Kim:2005-2, Kim-Kim:2008-1, Bao:2015-1, Kim-Kim:2018-1}, we give an obstruction for a knot $K$ to being in $\cT^{\Delta_K}$ (see Theorem~\ref{theorem:metabelian-obstruction}), and using it we show that the knot $R_1(J,D)$ with $J:=D_+(T_{2,3},0)\# -T_{2,3}$ is nontrivial in $\cT/\cT^{\lambda_1}$ (see Proposition~\ref{proposition:example}). It is unknown to the authors whether or not the knot $R_1(J,D)$  is in $\cT_1$.

The rest of the paper is organized as follows. In Section~\ref{section:2}, we introduce the bipolar filtration and prove Theorem~\ref{theorem:d-invariant-satellite-intro}, and in Section~\ref{section:proof of Theorem A} we prove Theorem~\ref{theorem:main}. In Appendix, we prove Theorem~\ref{theorem:metabelian-obstruction}.

In this paper, all homology groups are understood with integer coefficients unless mentioned otherwise. 

\subsubsection*{Acknowledgements}
The authors thank Jae Choon Cha for sharing with them his insight on the subject and his early drafts of \cite{Cha:2019-1}.

\section{The bipolar filtration, satellite links and $\nu^+$-equivalence}\label{section:2}
\subsection{The bipolar filtration}\label{subsection:bipolar-filtration}
In this subsection, we recall the definition of the bipolar filtration $\{\cT_n\}$ introduced in \cite{Cochran-Harvey-Horn:2012-1}. Throughout this paper, for a knot $K$ we denote by $M(K)$ the zero-framed surgery on $K$ in $S^3$. For a group $G$, we define $G^{(0)}:=G$ and $G^{(n+1)}:=[G^{(n)},G^{(n)}]$ for each $n\ge 0$. 
\begin{definition}[{\cite[Definition~5.1]{Cochran-Harvey-Horn:2012-1}}]\label{definition:n-bipolar}
Let $n\ge 0$ be an integer. A knot $K$ is \emph{$n$-positive} if $M(K)$ bounds a compact connected oriented smooth 4-manifold $W$ which satisfies the following.
\begin{enumerate}
\item The inclusion-induced homomorphism $H_1(M(K))\to H_1(W)$ is an isomorphism and $\pi_1(W)$ is normally generated by a meridian of $K$.
\item $H_2(W)$ has a basis which consists of disjointly embedded compact connected oriented surfaces $S_i$, $1\le i\le r$, such that for each $i$, the surface $S_i$ has a normal bundle of Euler class 1 and $\pi_1(S_i)\subset \pi_1(W)^{(n)}$. 
\end{enumerate}
The 4-manifold $W$ is called an \emph{$n$-positon}. Similarly, by changing the Euler class condition from 1 to $-1$, we define an \emph{$n$-negative knot} and an \emph{$n$-negaton}. A knot is \emph{$n$-bipolar} if it is $n$-positive and $n$-negative. We let $\cT_n:=\{[K]\in \cT\,\mid \, K\textrm{ is $n$-bipolar}\}$.
\end{definition}

Each $\cT_n$ is a subgroup of $\cT$ and since for $m\ge n$, an $m$-bipolar knot is $n$-bipolar, we have a descending filtration
\[
0\subset \cdots \subset \cT_{n+1}\subset \cT_n\subset \cdots \subset \cT_1\subset \cT_0\subset \cT.
\]

\subsection{$\nu^+$-equivalence}
Let $\nu^+$ be the knot concordance invariant introduced by 
Hom and Wu \cite{Hom-Wu:2016-1}. Using the $\nu^+$-invariant, one can consider the following equivalence relation coarser than concordance. 
\begin{definition}\label{definiton:nu+equivalence}Two knots 
$K$ and $J$ in 
$S^3$ are \emph{$\nu^+$-equivalent} if 
\[\nu^+(K\# -J) =\nu^+(-K\# J)=0.\]
\end{definition}

\begin{remark}\label{remark:nu+-equivalence}  By \cite[Proposition~3.11]{Hom:2017-1}, the following conditions are equivalent (see also \cite[Lemma~3.1]{Kim-Park:2016-1}).
\begin{enumerate}
\item Two knots $K$ and $J$ are $\nu^+$-equivalent.
\item $CFK^\infty(K)$ and $CFK^\infty(J)$ are stably chain homotopy equivalent. More precisely,  there are two acyclic chain complexes $A$ and $A'$ such that 
$CFK^\infty(K)\oplus A$ and $CFK^\infty(J)\oplus A'$ are filtered chain homotopy equivalent. 
\item $d_{\pm 1/2}(S^3_0(K\# -J))=d_{\pm 1/2}(S^3_0(J\# -K))=\pm\frac{1}{2}$.
 \end{enumerate}
 
For example, $D_+(T_{2,3})\#-T_{2,3}$ is $\nu^+$-equivalent to the unknot since $CFK^\infty(D_+(T_{2,3}))$ is filtered chain homotopy equivalent to $CFK^\infty(T_{2,3})\oplus A$ for some acyclic chain complex $A$ by Proposition~6.1 of \cite{Hedden-Kim-Livingston:2016-1}.

\end{remark}

\subsection{Satellite links with $\nu^+$-equivalent companions}\label{subsection:satellite-links}
We recall Theorem~\ref{theorem:d-invariant-satellite-intro} below.

\begin{theorem}\label{theorem:d-invariant-satellite}Let $K$ and $J$ be knots and $P$ be an ordered $k$-component link in $S^1\times D^2$. Let $\Lambda=(n_1,\ldots,n_k)$ be surgery coefficients such that $S^3_{\Lambda}(P(K))$ and $S^3_{\Lambda}(P(J))$ are $\Z_2$-homology spheres. If $K$ and $J$ are $\nu^+$-equivalent, then 
\[d(S^3_{\Lambda}(P(K)),\mathfrak{t}_K)=d(S^3_{\Lambda}(P(J)),\mathfrak{t}_J)\]
 where $\mathfrak{t}_K$ and $\mathfrak{t}_J$ are $\operatorname{Spin}^{c}$ structures satisfying $c_1(\mathfrak{t}_K)=c_1(\mathfrak{t}_J)$ under the canonical bijection $H^2(S^3_{\Lambda}(P(K)))\to H^2(S^3_{\Lambda}(P(J)))$.
\end{theorem}
\begin{proof}We prove Theorem~\ref{theorem:d-invariant-satellite} assuming Proposition~\ref{proposition:semidefinite} below. Let $a\in G$ be the image of $c_1(\mathfrak{t}_K)$ via the isomorphism given in Proposition~\ref{proposition:semidefinite} (1). By Proposition~\ref{proposition:semidefinite} (3), $(0,a)\in \Z\oplus G\cong H^2(Z)$ is a characteristic element, so we can choose a Spin$^c$ structure $\mathfrak{s}$ on $Z$ such that 
$\mathfrak{s}$ restricts to $\mathfrak{t}_0$, $\mathfrak{t}_K$ and $\mathfrak{t}_J$ on $S^3_0(J\# -K)$, $S^3_{\Lambda}(P(K))$, and $S^3_{\Lambda}(P(J))$, respectively, where $\mathfrak{t}_0$ is the torsion Spin$^c$ structure on $S^3_0(J\# -K)$. Since $b_1(\d Z)=1$, $\d Z$ has trivial triple cup product, and hence  $\d Z$ has standard $HF^\infty$  by Theorem~3.2 of \cite{Levine-Ruberman:2014-1}. 
Since $Z$ is negative semidefinite, by Theorem~4.7 of \cite{Levine-Ruberman:2014-1},
\[0\leq 4d_{bot}(\d Z,\mathfrak{s}|_{\d Z})+2.\]
On the other hand, $(\d Z,\mathfrak{s}|_{\d Z})=( S^3_{\Lambda}(P(K)),\mathfrak{t}_K)\# (-S^3_{\Lambda}(P(J)),\mathfrak{t}_J)\# (S^3_0(J\# -K),\mathfrak{t}_0)$. By the additivity in  \cite[Proposition~4.3]{Levine-Ruberman:2014-1}, the inequality becomes
\[0\leq 4d( S^3_{\Lambda}(P(K)),\mathfrak{t}_K)-4d(S^3_{\Lambda}(P(J)),\mathfrak{t}_J)+4d_{-1/2}(S^3_0(J\# -K))+2.\]
Since $K$ and $J$ are $\nu^+$-equivalent, by Remark~\ref{remark:nu+-equivalence}, $d_{-1/2}(S^3_0(J\# -K))=-\frac{1}{2}$. It follows that 
\[d( S^3_{\Lambda}(P(K)),\mathfrak{t}_K)\geq d(S^3_{\Lambda}(P(J)),\mathfrak{t}_J).\] By changing the roles of $J$ and $K$, we obtain the desired equality. 
\end{proof}
\begin{proposition}\label{proposition:semidefinite} Under the same hypothesis as in Theorem~\ref{theorem:d-invariant-satellite}, there is a compact connected oriented smooth $4$-manifold $Z$ with $\d Z\cong S^3_{\Lambda}(P(K))\#-S^3_{\Lambda}(P(J))\# S^3_0(J\#-K)$ satisfying the following.

\begin{enumerate}
\item\label{item:spinc} $H^2(Z)\cong \Z \oplus G$ such that the $\Z$-summand is isomorphic to $H^2(S^3_0(J\# -K))$ and the group $G$ is isomorphic to both $H^2(S^3_{\Lambda}(P(K)))$ and $H^2(S^3_{\Lambda}(P(J)))$ via the inclusion-induced maps. The composition of the isomorphisms
\[H^2(S^3_{\Lambda}(P(K)))\to G\to H^2(S^3_{\Lambda}(P(J)))\]
coincide with the  canonical bijection $H^2(S^3_{\Lambda}(P(K)))\to H^2(S^3_{\Lambda}(P(J)))$.
\item\label{item:firsthomology} $H_1(Z)\cong G$.
\item\label{item:intersection}The intersection form of $Z$ is trivial. In particular, $Z$ is negative semidefinite.
\end{enumerate}
\end{proposition}

For an ordered, oriented link $P$ in $S^1\times D^2$, let $-P$ be the link in $S^1\times D^2$ which is the mirror of $P$ with the reversed string orientations. Denote by $H$ the genus 2 handlebody. Regard $H$ as the boundary connected sum of two copies of $S^1\times D^2$. Consider two links $P\sqcup -P$ and $C$ in $H$ in Figure~\ref{figure:links}. We need a lemma.

 \begin{figure}[htb!]
 \centering
  \includegraphics[scale=1]{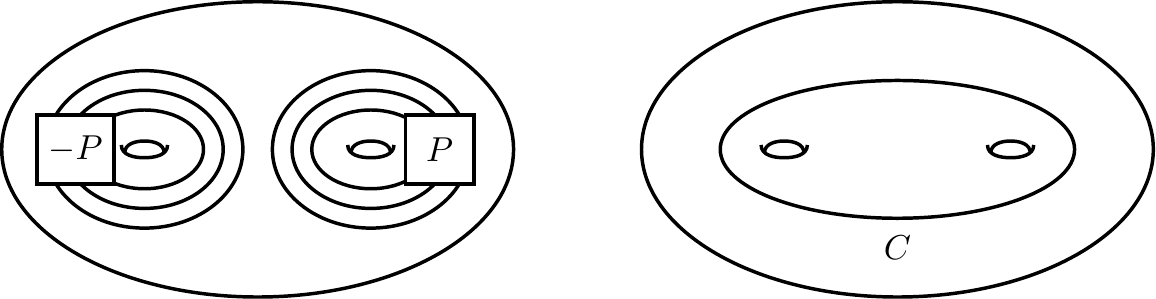}
  \caption{The links $P\sqcup - P$ and $C$ in $H$.}
  \label{figure:links}
  \end{figure}
\begin{lemma}\label{lemma:planar-cobordism} For each ordered, oriented $k$-component link $P$ in $S^1\times D^2$, there exist compact connected surfaces $S_1,\ldots,S_k$ of genus 0 with $\partial S_i\subset H\times \{0,1\}$ which are properly disjointly embedded in $H\times [0,1]$ such that the following hold\textup{:}
\begin{enumerate}

\item For each $i$, $\d S_i\cap H\times \{0\}$ is either the empty set or the union of finitely many parallel copies of $C$ endowed with a certain orientation.
\item For each $i$, $\d S_i\cap H\times \{1\}$ is the union of the $i$-th components of $P$ and $-P$. In particular, $(\bigsqcup_{i=1}^k \d S_i)\cap H\times \{1\}=P\sqcup -P$.
\end{enumerate} 
\end{lemma}

Before proving this lemma, we define a couple of terms: a \emph{diagram} of a link $L$ in $H$ is the projection of $L$ into the boundary connected sum of two annuli, a twice punctured disk, having transversal double singularities equipped with under and over passing information. Here, $H$ is assumed to be embedded in $\mathbb{R}^3$ in a standard way, whose projection is a twice punctured disk. The singularities of the diagram are called \emph{crossings}.

\begin{proof}[Proof of Lemma~\ref{lemma:planar-cobordism}]
Let $P_i$ be the $i$-th component of $P$ for $i=1,\ldots, k$. Let $M$ be the disk, where the boundary connected sum of $H=S^1\times D^2\, \natural\, S^1\times D^2$ takes place. One may consider $M$ as the mirror reflecting $P$ onto $-P$ in $H$. In other words, there is an involution $\iota$ on $H$ fixing each point of $M$ and $\iota(P_i)=-P_i$ for $i=1,\ldots,k$.

For each $i=1,\ldots,k$, we can choose a simple path $\gamma_i$ from  a point of $P_i$ to a point of $-P_i$ so that $\iota(\gamma_i)=\gamma_i$, $\gamma_i$ meets each of $M$, $P_i$, and $-P_i$ in exactly one point, but $\gamma_i$ does not meet $P_j$ or $\gamma_j$ for any $j\ne i$. We take band sums of $P_i$ with $-P_i$ along $\gamma_i$ for all~$i$. This procedure gives us a link $L$ in $H$ satisfying the following:
\begin{itemize}
 \item[(i)] Each component is of the form $K\# \iota(K)$ for some knot $K$ in $S^1\times D^2$.
 \item[(ii)] Each component intersects the mirror $M$ in exactly two points.
\end{itemize}
It also produces properly, disjointly embedded, compact connected surfaces $S'_1,\ldots,S'_k$ of genus 0 in $S^3\times [\frac{1}{2},1]$ such that $\partial S_i'\cap S^3\times \{1\}=P_i\sqcup -P_i$ and $\partial S_i'\cap S^3\times \{\frac{1}{2}\}=P_i\# -P_i$. Note that, if a link $L$ in $H$ satisfies conditions (i) and (ii), then, for every crossing $c$ of $L$, the overpassing arcs of $c$ and $\iota(c)$ belong to the same component of $L$ and the underpassing arcs of $c$ and $\iota(c)$ belong to the same component of $L$.

\begin{figure}[htb!]
\includegraphics{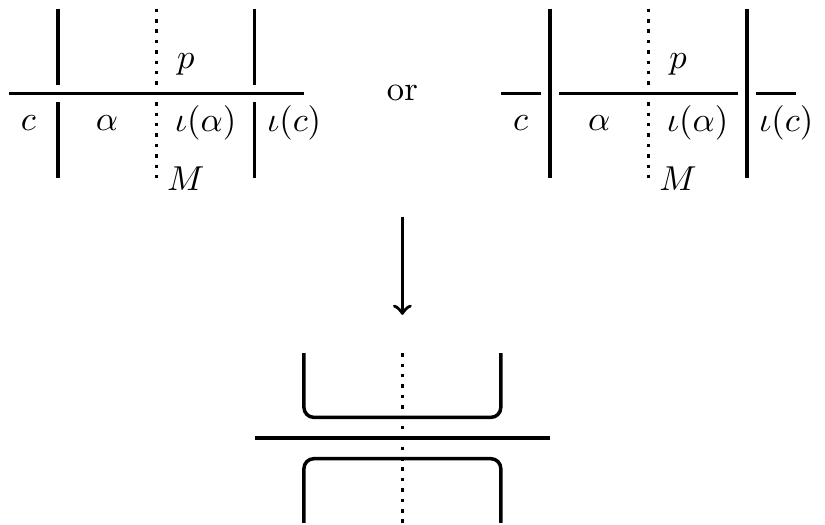}

 \caption{Band move.}
 \label{fig:band move}
\end{figure}
Beginning with a diagram $D$ of such a link $L$, we shall find a sequence of band sums, each of which decreases the number of crossings and increases the number of components. So, the final link consists of unlinked unknots and unlinked parallel copies of $C$. The traces of the band sums and the disks bounding the unlinked unknots give disjointly embedded, compact connected surfaces $S_1'',\ldots,S_k''$ of genus 0 in $S^3\times [0,\frac{1}{2}]$ such that $\partial S_i''\cap S^3\times \{\frac{1}{2}\}=P_i\# -P_i$ and $\partial S_i''\cap S^3\times \{0\}$ is either the empty set or the union of finitely many parallel copies of $C$. By taking $S_i$ to be the union of $S_i'$ and $S_i''$, this will prove the lemma.

We use induction on the number of crossings in $D$. Suppose $D$ has no crossing. Then, each component is either parallel to $C$ or unknotted in $H$. Use the innermost circle argument to cap off the unknotted components, and we are done for this case. Now assume that $D$ has a crossing. Let $K$ be a knot in $S^1\times D^2$ for which $K\#\iota(K)$ is a component of $D$ having a crossing. As stated above, the component meets $M$ at a point~$p$. Let $c$ be the crossing in $K$ from which to $p$ the part of $K$ has no crossings. Denote the part by $\alpha$. Then $\alpha\cup \iota(\alpha)$ is a path from $c$ to $\iota(c)$, which is a part of $K$ with no crossings. See Figure~\ref{fig:band move}. Along this path $\alpha\cup \iota(\alpha)$, we take a band move between the two arcs that transversally cross $\alpha\cup \iota(\alpha)$ at $c$ and $\iota(c)$. It is obvious that this band sum reduces the number of crossings and increases the number of components. The resulting link has two less crossings and still satisfies the above two conditions (i) and (ii). Now, inductively we can find the desired sequence as we claimed above.
\end{proof}

We are ready to prove Proposition~\ref{proposition:semidefinite}.
\begin{proof}[Proof of Proposition~\ref{proposition:semidefinite}] For given two knots $K$ and $J$ in $S^3$, Figure~\ref{figure:embedding} describes an embedding of $H\to S^3$. Let $f\colon H\times [0,1]\to S^3\times [0,1]$ denote the product of the embedding $H\to S^3$ and the identity on $[0,1]$.  For a given $k$-component link $P$ in $S^1\times D^2$, by Lemma~\ref{lemma:planar-cobordism}, there are connected, disjointly embedded surfaces $S_1,\ldots,S_k$ of genus 0 in $H\times [0,1]$ such that $\d S_i \cap H\times \{1\}$ is the union of the $i$-th component of $P$ and $-P$ for all $i$, and $\d S_i\cap H\times \{0\}$ is either empty or the union of finitely many parallel copies of $C$ endowed with a certain orientation. By the definition of $f$, $f(S_1),\ldots,f(S_k)$ are disjointly embedded surfaces of genus 0 in $S^3\times [0,1]$ such that $(\bigsqcup_{i=1}^k f(S_i))\cap S^3\times \{1\}$ is the split union $(P(K)\sqcup -P(J))\times \{1\} $ and $(\bigsqcup_{i=1}^k f(S_i))\cap S^3\times \{0 \}$ is the union of finitely many parallel copies of $K\# -J$. 

\begin{figure}[htb!]
\includegraphics{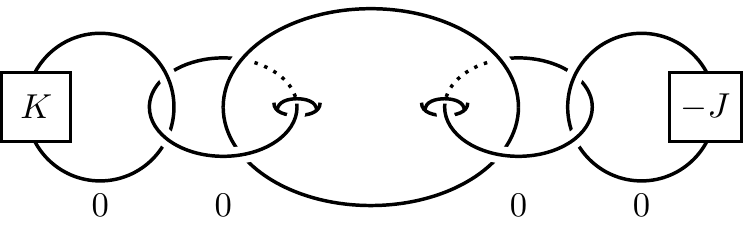}

 \caption{An embedding of $H$ in $S^3$.}
 \label{figure:embedding}
\end{figure}
Let $W$ be the 4-manifold obtained from $S^3\times [0,1]$ by attaching a 0-framed 2-handle along $(K\# -J)\times \{0\}$. Note that $P(K)\times \{1\}$ and $-P(J)\times \{1\}$ cobound disjointly embedded annuli $A_1,\ldots,A_k$ in $W$. (To obtain $A_i$, cap off $f(S_i)$ by adding finitely many parallel copies of the 2-handle core disk.) Let $Z_0$ be the result of $\Lambda=(n_1,\ldots,n_k)$-surgery on $W$ along the annuli $A_1,\ldots,A_k$. Note that $Z_0$ has two (oriented) boundary components $S^3_\Lambda(P(K))\# -S^3_{\Lambda}(P(J))$ and $S^3_0(J\#-K)$. The desired 4-manifold $Z$ is obtained from $Z_0$ by deleting a neighborhood of an arc joining the boundary components of $Z_0$. Homologically, $Z$ is obtained from $D^2\times S^2$ by doing $\Lambda=(n_1,\ldots,n_k)$-surgeries on $k$ annuli, so the desired properties of $Z$ follow.
\end{proof}
By the Akbulut-Kirby method \cite{Akbulut-Kirby:1979-1}, we obtain surgery diagrams for the $r$-fold cyclic branched covers $\Sigma_r(R_m(J,D))$ and $\Sigma_r(R_m(U,D))$ of $S^3$ along $R_m(J,D)$ and $R_m(U,D)$ from the genus 1 Seifert surface in Figure~\ref{figure:seed-knot}. The surgery link of $\Sigma_r(R_m(J,D))$ is obtained from that of $\Sigma_r(R_m(U,D))$ by doing iterated satellite operations (actually $r$ times) with the fixed companion knot $J$. As in Theorem~\ref{theorem:d-invariant-satellite}, we can identify the Spin$^c$ structures of $\Sigma_r(R_m(J,D))$ and those of $\Sigma_r(R_m(U,D))$. We remark that $\Sigma_r(R_m(J,D))$ and $\Sigma_r(R_m(U,D))$ are $\Z_2$-homology 3-spheres since 
\[H_1(\Sigma_r(R_m(J,D)))\cong H_1(\Sigma_r(R_m(U,D)))\cong \Z_N\oplus \Z_N\]
 where $N=(m+1)^r-m^r$ (see \cite[Proposition~2]{Gilmer-Livingston:1992-2}).

\begin{theorem}\label{theorem:d-invariant-practical}If $J$ is $\nu^+$-equivalent to the unknot $U$, then 
\[d(\Sigma_r(R_m(J,D)),\mathfrak{s})=d(\Sigma_r(R_m(U,D)),\mathfrak{s}).\]
\end{theorem}
\begin{proof} From the discussion given in the paragraph preceding Theorem~\ref{theorem:d-invariant-practical}, we obtain the conclusion by applying Theorem~\ref{theorem:d-invariant-satellite} $n$ times.
\end{proof}

\section{Proof of Theorem~\ref{theorem:main}}\label{section:proof of Theorem A}In this section we prove Theorem~\ref{theorem:main}. The proof of Theorem~\ref{theorem:main} is essentially identical with the proof of \cite[Theorem~D]{Cha:2019-1} with some modifications. Below we give a list of the key modifications for applying the proof of  \cite[Theorem~D]{Cha:2019-1} to the case $n=1$, which gives the proof of  Theorem~\ref{theorem:main}: 
\begin{enumerate}
\item Use different $J_i$, as defined in the introduction of this paper, for the construction of $R_m(J_i, D)$.
\item For the computation of the von Neumann $\rho$-invariants in \cite[Section~3]{Cha:2019-1}, in \cite[Definition~3.1]{Cha:2019-1} take $n=1$ and replace the coefficients $\Z_{p_1}$  by $\Q$. This is due to our choice of $J_i$.
\item For the computation of correction terms $d(\Sigma_r(R_m(J_i,D)),\mathfrak{s})$ in \cite[Section~4]{Cha:2019-1}, use Theorem~\ref{theorem:d-invariant-practical} in place of \cite[Lemma~8.2]{Cochran-Harvey-Horn:2012-1} and \cite[Lemma~5.1]{Cha-Kim:2017-1}, which is again due to our choice of $J_i$. This is our key technical contribution for the proof of Theorem~\ref{theorem:main}.
\end{enumerate}

The above key modifications with some minor changes for the proof of Theorem~D in \cite{Cha:2019-1} will easily produce the proof of Theorem~\ref{theorem:main}. 

Nonetheless, to clarify the proof, we give a more detailed sketch of the proof of Theorem~\ref{theorem:main} below following the arguments in the proof of Theorem~D in \cite{Cha:2019-1}. In particular, we do not include some arguments in the proof of \cite[Theorem~D]{Cha:2019-1} which are not necessary for our proof but were needed for the proof of  \cite[Theorem~D]{Cha:2019-1} to take care of far more general cases $n>1$. 
\bigskip

For the rest of Section~\ref{section:proof of Theorem A}, we give a sketch of the proof of Theorem~\ref{theorem:main}. Fix an odd positive integer $m$ and let $R:=R_m$. For each $i\ge 1$ let $K_i:=R(J_i,D)$ where $J_i=D_-(k_i T_{2,3}, 2k_i)$ as defined in the introduction.

The following lemma gives the key properties of $J_i$. For a knot $K$, we denote by $\sigma_K$ the Levine-Tristram signature function of $K$ defined on $S^1\subset \C$.

\begin{lemma}\label{lemma:J_i}
	The knots $J_i$ satisfy the following.
	\begin{enumerate}
		\item Each $J_i$ is $0$-negative.
		\item For $i\ge 1$, $\int_{S^1}\sigma_{J_i}(\omega)\,\,d\omega$ are linearly independent over $\Q$.
		\item Each $J_i$ is $\nu^+$-equivalent to the unknot.
	\end{enumerate}
\end{lemma}
\begin{proof}Since $J_i$ can be unknotted by changing a negative crossing, 
	Property~(1) follows from \cite[Proposition~3.1]{Cochran-Harvey-Horn:2012-1}.  Since $
	\Delta_{J_i}=2k_it-(4k_i-1)+2k_it^{-1}$, Property~(2) follows from the proof of Proposition~2.6 in \cite[p.\ 121]{Cochran-Orr-Teichner:2004-1}. For the concordance invariant $\tau$ of Ozsv\'{a}th and Szab\'{o} \cite{Ozsvath-Szabo:2003-1}, Sato \cite[Theorem~1.2]{Sato:2019-1} showed that every genus 1 knot with $\tau=0$ is $\nu^+$-equivalent to the unknot. So, it is enough to show $\tau(J_i)=0$ to see  Property~(3). Note that $-J_i=D_+(k_iT_{2,{-3}}, {-2}k_i)$. Since $\tau(k_iT_{2,{-3}})={-2}k_i$, we have $\tau(-J_i)=0$ by \cite[Theorem~1.5]{Hedden:2007-1}, and hence $\tau(J_i)=-\tau(-J_i)=0$.
\end{proof}

Our knots $K_i$ are topologically slice and 1-bipolar, that is, $K_i\in \cT_1$ for all $i\ge 1$; since $J_i$ is 0-negative (and $D$ is 0-positive), this follows from \cite[Lemma~2.3]{Cha-Kim:2017-1}. 

Finally, we show that  if $K$ is a nontrivial linear combination of $K_i$, then $K$ is not concordant to any knot with Alexander polynomial coprime to $\lambda_m$, modulo $\cT_2$. Let $L$ be a knot with Alexander polynomial coprime to $\lambda_m$. For integers $a_i$, let $K:= (\#_{i=1}^ra_iK_i)\# L$. By renumbering indices and changing the orientation if necessary, we may assume that each $a_i$ is nonzero and $a_1>0$. Now it suffices to show that $K$ is not 2-bipolar. 

Suppose $K$ is 2-bipolar. Then in \cite[Subsection~2.3]{Cha-Kim:2017-1}, a 1-negaton, denoted by  $X^-$, with boundary $M(K_1)$ is constructed such that  
for
\[
P:=\Ker\{i_*\colon H_1(M(K_1);\Q[t^{\pm 1}])\to H_1(X^-;\Q[t^{\pm 1}])\}
\] 
where $i_*$ is induced from the inclusion, either $P=\langle \alpha_J\rangle$ or $P=\langle \alpha_D\rangle$. 

\bigskip

\noindent{\bf Case 1: $P=\langle \alpha_D \rangle$.} In this case we follow the arguments in \cite[Section~3]{Cha:2019-1}; the 4-manifold $X^-$ is modified  to a new 4-manifold $X^0$, keeping the boundary, defined to be
\[
X^0:= V^0\cup_{\partial_-C} C \cup_{\partial_+C\sm M(K_1)} \left(\left(a_1-1\right)Z_1^0 \sqcup\left(\bigsqcup_{i>1}|a_i|Z_i^0\right)\sqcup Z_L^0\right).
\]
See Figure~\ref{figure:solution}. We refer the reader to Section~3 in \cite{Cha:2019-1} for the notations in the definition of $X^0$.

\begin{remark}
	The above $X^0$ is not the same as $X=X_0$ defined in \cite[Section~3]{Cha:2019-1} for the case $n=1$; it is the same as $X$ \emph{with $E_0$ removed}, which is also the same as $X_1$ for $n=1$ in \cite[Section~3]{Cha:2019-1}. See Figure~5 in\cite{Cha:2019-1}. 
\end{remark}

\begin{figure}
\includegraphics[scale=0.8]{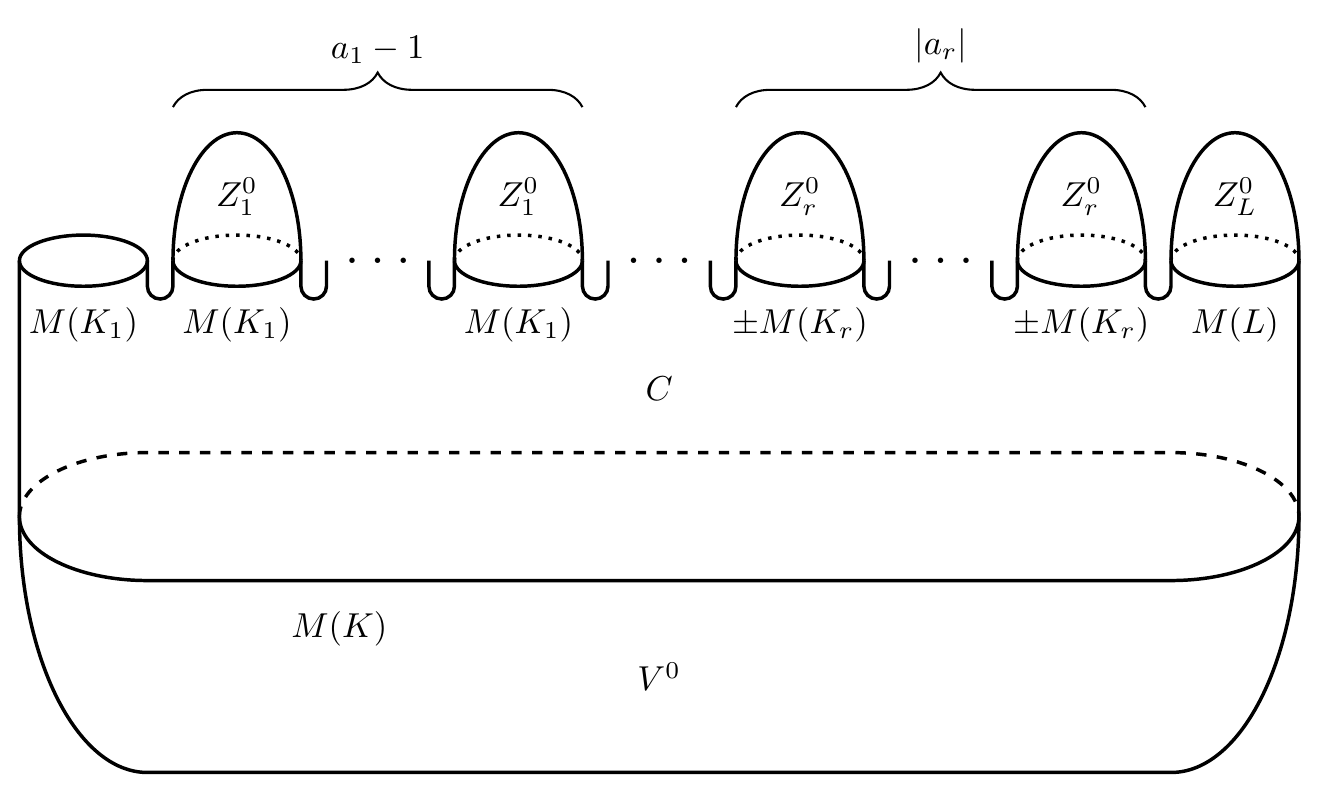}
\caption{The 4-manifold $X^0$.}\label{figure:solution}
\end{figure}

In \cite[Section~3]{Cha:2019-1} it is shown that $\pi_1(X^0)\cong \pi_1(X^-)$ and $H_1(X^0;\Q[t^{\pm 1}])\cong H_1(X^-;\Q[t^{\pm 1}])$. It is also shown that
\[
P=\Ker\{i_*\colon H_1(M(K_1);\Q[t^{\pm 1}])\to H_1(X^0;\Q[t^{\pm 1}])\},
\]
which is $\langle \alpha_D\rangle$  by our hypothesis. 

We will use $X^0$ and its subspaces to compute the von Neumann $\rho$-invariant of $M(K_1)$ defined in \cite{Cheeger-Gromov:1985-1}, and for more details of the von Neumann $\rho$-invariant we refer the reader to \cite[Subsection~3.2]{Cha:2019-1}. We will use the same notations as in \cite{Cha:2019-1}; suppose $M$ is a closed oriented 3-manifold and $W$ is a 4-manifold with $\partial W=M$. Suppose for a countable discrete group $\Gamma$ there is a homomorphism $\phi\colon \pi_1(M) \to \Gamma$ which extends to $\pi_1(W)$. Then, the von Neumann $\rho$-invariant of $M$ associated to $\phi$ is 
\[
\rho(M,\phi)=\osigmat_\Gamma(W):=\sign^{(2)}_\Gamma(W)-\sign(W).
\]

For our proof, we choose a group $\Gamma$ and a representation $\phi\colon \pi_1(X)\to \Gamma$ following Definition~3.1 in \cite{Cha:2019-1}, wherein we take $n=1$ and replace $\Z_{p_1}$ by $\Q$. This is one of our key modifications for the proof of \cite[Theorem~D]{Cha:2019-1}. More details follow: let 
\[
\Sigma:=\{f(t)\in \Q[t^{\pm 1}]\,\mid\, f(1)\ne 0,\, \textrm{gcd}(f(t), \lambda_m)=1\}.
\]
Let $\pi:=\pi_1(X^0)$ and $\cP^1\pi:=\pi^{(1)}=[\pi,\pi]$. Let $\cP^2\pi$ be the kernel of the composition
\begin{multline*}
\cP^1\pi\to \cP^1\pi/[\cP^1\pi,\cP^1\pi]\otimes_\Z\Q =H_1(\pi;\Q[t^{\pm 1}])\to H_1(\pi;\Q[t^{\pm 1}]\Sigma^{-1})\\
\to H_1(X;\Q[t^{\pm 1}]\Sigma^{-1}) \to H_1(X;\Q[t^{\pm 1}]\Sigma^{-1})/\Im H_1(Z_L^0;\Q[t^{\pm 1}]\Sigma^{-1}),
\end{multline*}
where $\Im H_1(Z_L^0;\Q[t^{\pm 1}]\Sigma^{-1})$ is the image of $H_1(Z_L^0;\Q[t^{\pm 1}]\Sigma^{-1})$ under the homomorphism induced from the inclusion from $Z_L^0$ to $X^0$.

Now we let $\Gamma:=\pi/\cP^2\pi$ and let $\phi\colon \pi_1(X^0)=\pi\to \Gamma$ be the quotient map. By abuse of notation, we denote by $\phi$ the restrictions of $\phi$ to subspaces of $X^0$.

By Novikov additivity, we have the following equation (compare it with Equation~(3.2) in \cite{Cha:2019-1}).
\begin{equation}\label{equation:signature-of-X}
	\rho(M(K_1),\phi) =  \osigmat_\Gamma(V^0) + \osigmat_\Gamma(C) + \sum_{j=1}^{a_1-1}\osigmat_\Gamma(Z_{1,j}^0) + \sum_{i=2}^r\sum_{j=1}^{|a_i|}\osigmat_\Gamma(Z_{i,j}^0) + \osigmat_\Gamma(Z_L^0),
\end{equation}
where $Z_{i,j}^0$ are copies of $Z_i^0$. We compute each term of the above equation, and this will lead us to a contradiction. 
\bigskip

(1) $\rho(M(K_1),\phi) = \int_{S^1}\sigma_{J_1}(\omega)\,\,d\omega$. 
To show this, we construct an integral 1-solution $Z$ for $M(K_1)$ following the arguments in \cite{Cochran-Orr-Teichner:2004-1} (and also \cite{Cha:2010-01}). The knot $R(U,D)$ where $U$ denotes the unknot is a slice knot whose slice disk is obtained by cutting the band dual to $\alpha_D$. Let $W$ be the exterior of the slice disk in the 4-ball, and hence $\partial W=R(U,D)$. Since $J_1$ is 0-negative, it is also integrally 0-solvable by \cite[Proposition~5.5]{Cochran-Harvey-Horn:2012-1}.  Let $W'$ be an integral 0-solution with $\partial W'=M(J_1)$. By doing surgery along the curves generating $\pi_1(W')^{(1)}$ we may assume that $\pi_1(W')\cong \Z$. Consider $M(J_1)$ as $(S^3\sm \nu(J_1))\cup (S^1\times D^2)$ where $\nu(J_1)$ denotes the open tubular neighborhood of $J_1$. Now let $Z$ be the 4-manifold obtained from $W$ and $W'$ by identifying the tubular neighborhood of the curve $\alpha_J$ in $M(R(U,D))=\partial W$ and $S^1\times D^2\subset M(J_1)=\partial W'$. Then $Z$ is an integral 1-solution with $\partial Z=M(K_1)$.

Let 
\begin{align*}
\Gamma'&:=\left(H_1(X^0;\Q[t^{\pm 1}]\Sigma^{-1})/\Im H_1(Z_L^0;\Q[t^{\pm 1}]\Sigma^{-1})\right)\rtimes \Z\\
		& \cong H_1(X^0\setminus Z_L^0;\Q[t^{\pm 1}]\Sigma^{-1})\rtimes \Z
\end{align*}
where the last isomorphism is given by \cite[Subsection~3.1]{Cha:2019-1}. Here, in the semidirect products $\Z=H_1(X^0)$ acts on the left summands via deck transformations. Due to the definition of $\Gamma$, there is an injective homomorphism $i\colon \Gamma \hookrightarrow \Gamma'$.

We assert that the composition $i\circ\phi\colon \pi_1(M(K_1))\to \Gamma \hookrightarrow \Gamma'$ extends to $\pi_1(Z)$.  Recall that 
\[
\langle \alpha_D \rangle=\Ker\{H_1(M(K_1);\Q[t^{\pm 1}])\to H_1(X^0;\Q[t^{\pm 1}])\}.
\]
Due to the definition of $\Sigma$ and the fact that $\Delta_L$ is coprime to $\Delta_{K_1}$, using Mayer-Vietoris sequences one can easily see that 
\[
\langle \alpha_D \rangle = \Ker\{H_1(M(K_1);\Q[t^{\pm 1}]) \to H_1(X^0\setminus Z_L^0;\Q[t^{\pm 1}]\Sigma^{-1})\}.
\]
Therefore, the map $i\circ \phi \colon \pi_1(M(K_1))\to \Gamma \hookrightarrow \Gamma'$ factors through the  injective map 
\[
\bar{\phi}\colon \left(H_1(M(K_1);\Q[t^{\pm 1}])/\langle\alpha_D \rangle\right)\rtimes \Z\to \Gamma'.
\]

Due to the construction of $W$, one can easily see that 
\[
H_1(W;\Q[t^{\pm 1}])\cong H_1(R;\Q[t^{\pm 1}])/\langle \alpha_D\rangle\cong H_1(M(K_1);\Q[t^{\pm 1}])/\langle \alpha_D \rangle.
\]
Using Mayer-Vietoris sequences, one can see that $H_1(Z;\Q[t^{\pm 1}])\cong H_1(W;\Q[t^{\pm 1}])$, and therefore $H_1(Z;\Q[t^{\pm 1}])\cong H_1(M(K_1);\Q[t^{\pm 1}])/\langle \alpha_D\rangle$. It follows that the map defined to be the composition
\begin{multline*}
\pi_1(Z)\to \pi_1(Z)/\pi_1(Z)^{(2)}\cong H_1(Z;\Z[t^{\pm 1}])\rtimes \Z\\ 
\to H_1(Z;\Q[t^{\pm 1}])\rtimes \Z \cong (H_1(M(K_1);\Q[t^{\pm 1}])/\langle \alpha_D\rangle)\rtimes \Z \xrightarrow{\bar{\phi}} \Gamma'
\end{multline*}
is an extension of $i\circ\phi \colon \pi_1(M(K_1))\to \Gamma'$, which is what we asserted. By abuse of notation we also denote by $\bar{\phi}$ the extension to $\pi_1(Z)$ and its restrictions to subspaces of $Z$.  

By the subgroup property of the $\rho$-invariant (see Property~(2.3) in \cite[p. 108]{Cochran-Orr-Teichner:2004-1}), $\rho(M(K_1),\phi)=\rho(M(K_1),i\circ \phi)$.  Therefore, we have $\rho(M(K_1),\phi) =\osigmat_{\Gamma'}(Z)$, where $\osigmat_{\Gamma'}(Z) =\osigmat_{\Gamma'}(W) + \osigmat_{\Gamma'}(W')$ by Novikov additivity. Since $W$ is a slice disk exterior, $\osigmat_{\Gamma'}(W)=0$ by \cite[Theorem~4.2]{Cochran-Orr-Teichner:2003-1}. Since the curve $\alpha_J$ is of infinite order in $H_1(M(K_1);\Q[t^{\pm 1}])/\langle \alpha_D \rangle$, the subgroup of $\Gamma'$ generated by $\bar{\phi}(\alpha_J)$ is isomorphic to $\Z$. Let $\mu_{J_1}$ be a meridian of $J_1$. By the identification of $W$ and $W'$, we have $\bar{\phi}(\mu_{J_1})=\bar{\phi}(\alpha_J)$. Since $\bar{\phi}$ on $\pi_1(M(J_1))$ factors through $\pi_1(W')\cong \Z$ and $\bar{\phi}(\mu_{J_1})$ is of infinite order in $\Gamma'$, the image of $\pi_1(M(J_1))$ under $\bar{\phi}$ is isomorphic to $\Z$. Therefore by \cite[Proposition~5.13]{Cochran-Orr-Teichner:2003-1} and \cite[Lemma~5.3]{Cochran-Orr-Teichner:2004-1}, we obtain $\osigmat_{\Gamma'}(W')= \rho(M(J_1),\bar{\phi})=\int_{S^1}\sigma_{J_1}(\omega)\,\,d\omega$.

(2) $\osigmat_\Gamma(V^0)=\osigmat_\Gamma(C)=0$. This is proved in \cite[Subsection~3.2]{Cha:2019-1}.

(3) $\osigmat_\Gamma(Z_{1,j}^0)=0$ or $-\int_{S^1}\sigma_{J_i}(\omega)\,\,d\omega$,  and for $i\ge 2$,  $\osigmat_\Gamma(Z_{i,j}^0)=0\mbox{ or } \pm \int_{S^1}\sigma_{J_i}(\omega)\,\,d\omega$. This follows from \cite[Lemma~3.3]{Cha-Kim:2017-1}. (Lemma~3.3 in \cite{Cha-Kim:2017-1} uses $\Z_p$ coefficients, and its proof works with $\Z$ coefficients as well.) Or, one can also prove it using arguments similar to the one in $(1)$ above. 

(4) $\osigmat_\Gamma(Z_L^0)=0$. This is a key ingredient of the proof of \cite[Theorem~D]{Cha:2019-1}, and it is proved in \cite[Subsection~3.2]{Cha:2019-1}. For the reader's conveience, and since we use $\Q$ coefficients in place of $\Z_p$ coefficients which was used in \cite{Cha:2019-1}, we give a brief sketch of the computation. Due to the definition of $\cP^2\pi_1(X^0)$, the map $\phi$ on $\pi_1(Z_L^0)$ maps $\pi_1(Z_L^0)^{(1)}$ trivially into $\Gamma$. Therefore, the map $\phi$ on $\pi_1(Z_L^0)$ factors through the injective map $H_1(Z_L^0)\to \Gamma$, and therefore $\osigmat_\Gamma(Z_L^0) =\int_{S^1}\sigma_L(\omega)\,\,d\omega$ by \cite[Proposition~5.13]{Cochran-Orr-Teichner:2003-1} and \cite[Lemma~5.3]{Cochran-Orr-Teichner:2004-1}. Since $L$ is concordant to the connected sum of $K$ and $-a_iK_i$ and each of $K$ and $K_i$ is 1-bipolar, $L$ is also 1-bipolar. Since $L$ is 1-bipolar, it is algebraically slice by \cite[Corollary~5.7]{Cochran-Harvey-Horn:2012-1}, and hence $\int_{S^1}\sigma_L(\omega)\,\,d\omega=0$.
\bigskip

From Equation~(\ref{equation:signature-of-X}) and the above computations (1)--(4), we conclude that for some $\epsilon_i\in \Z$ where $\epsilon_1>0$, 
\[
\sum_{i=1}^r\left(\epsilon_i\cdot \int_{S^1}\sigma_{J_i}(\omega)\,\,d\omega\right)=0.
\]
But it contradicts that $\int_{S^1}\sigma_{J_i}(\omega)\,\,d\omega$ are linearly independent over $\Z$ (see Lemma~\ref{lemma:J_i}~(2)). 

\bigskip

\noindent{\bf Case 2: $P=\langle \alpha_J \rangle$.} In this case the proof is exactly the same as the proof for Case $P=\langle \alpha_J \rangle$ of Theorem~D in \cite[Section~4]{Cha:2019-1} except that we use Theorem~\ref{theorem:d-invariant-practical}, which is our key technical contribution in this case: let $K_0:=R(U,D)$. Let $\Sigma_r(K_1)$ (resp. $\Sigma_r(K_0)$) be the $r$-fold cyclic cover of $S^3$ branched along $K_1$ (resp. $K_0$). In \cite[Section~4]{Cha:2019-1} it was asserted that  $d(\Sigma_r(K_0), \mathfrak{s}_{\Sigma_r}+k\cdot \hat{x_1})\ge 0$ for all sufficiently large prime $r$ and for all $k\in \Z$  (see Lemma~4.1 in \cite{Cha:2019-1} and the paragraphs preceding Lemma~4.1) due to the fact that $J^1_{n-1}$ is unknotted by changing some positive crossings to negative crossings. 

Note that our $J_1$ cannot be unknotted by changing positive crossings to negative crossings. But since $J_1$ is $\nu^+$-equivalent to the unknot (see Lemma~\ref{lemma:J_i}~(3)), by Theorem~\ref{theorem:d-invariant-practical}, 
\begin{equation}\label{equation:d-invariant}
d(\Sigma_r(K_1), \mathfrak{s}_{\Sigma_r}+k\cdot \hat{x_1})=d(\Sigma_r(K_0), \mathfrak{s}_{\Sigma_r}+k\cdot \hat{x_1})
\end{equation}
for all prime $r$ and for all $k\in \Z$. Since $M(K_1)$ bounds a 1-negaton $X^-$ and 
\[
\langle \alpha_J\rangle =\Ker\{H_1(M(K_0);\Q[t^{\pm 1}])\to H_1(W;\Q[t^{\pm 1}])\},
\]
using the same arguments as in the proof of \cite[Lemma~4.1]{Cha:2019-1}, we obtain 
\[
d(\Sigma_r(K_1), \mathfrak{s}_{\Sigma_r}+k\cdot \hat{x_1})\ge 0
\] 
for all $k\in \Z$ and for all sufficiently large prime $r$. Now by Equation~(\ref{equation:d-invariant}) above, we obtain the desired conclusion of Lemma~4.1 in \cite{Cha:2019-1}. The rest of the proof for this case is exactly identical with the proof given in \cite[Section~4]{Cha:2019-1}, and this completes the sketch of the proof of Theorem~\ref{theorem:main}.

\appendix
\setcounter{secnumdepth}{0}

\section{Appendix: metabolizers for Blanchfield forms and concordance invariants}\label{section:appendix}

\setcounter{secnumdepth}{3}
\setcounter{theorem}{0}
\setcounter{equation}{0}
\def\thesection{A}
\def\thesubsubsection{\S\arabic{subsubsection}}

In Appendix, using metabelian concordance invariants given via metabolizers of the Blanchfield form of a knot, we give an obstruction for a knot $K$ to being concordant to any knot with Alexander polynomial coprime to $\Delta_K$ (see Theorem~\ref{theorem:metabelian-obstruction}). This obstruction is given by combining the results in \cite{Kim:2005-2, Kim-Kim:2008-1, Bao:2015-1, Kim-Kim:2018-1}. 

We review the invariants and notations which will be used in Theorem~\ref{theorem:metabelian-obstruction} below. Recall that $\Z=\langle t\rangle$ acts on $\Q(t)/\Q[t^{\pm 1}]$ via multiplication. This action induces a semidirect product $\Gamma:=(\Q(t)/\Q[t^{\pm 1}])\rtimes \Z$. For a knot $K$, let $B\ell_\Q$ be the rational Blanchfield form 
\[
B\ell_\Q\colon H_1(M(K);\Q[t^{\pm 1}])\times H_1(M(K);\Q[t^{\pm 1}]) \to \Q(t)/\Q[t^{\pm 1}].
\]
A choice of $x\in H_1(M(K);\Q[t^{\pm 1}])$ induces a homomorphism $\phi_x\colon \pi_1(M(K))\to \Gamma$ defined by $\phi_x(y) = (B\ell_\Q(x,y\mu^{-\epsilon(y)}), \epsilon(y))$ where $\epsilon\colon \pi_1(M(K_1)))\to \Z$ is the abelianization and $\mu$ is a meridian of $K$. Then, one can obtain the von Neumann $\rho$-invariant 
$\rho(M(K), \phi_x)$, whose definition can be found in Section~\ref{section:proof of Theorem A}.  (The homomorphism $\phi_x$ was introduced in \cite{Cochran-Orr-Teichner:2003-1}, wherein our $\Gamma$ is denote by $\Gamma_1$ or $\Gamma_1^U$ and called the \emph{rationally universal 1-solvable group}.)

For a rational homology 3-sphere $Y$ and a $\Spinc$ structure $\mathfrak{s}$ on $Y$, we let 
\[
\bar{d}(Y,\mathfrak{s}):=d(Y,\mathfrak{s})-d(Y,\mathfrak{s}_0)
\] 
where $d$ is the correction term invariant of Ozsv\'{a}th and Szab\'{o} \cite{Ozsvath-Szabo:2003-2} and $\mathfrak{s}_0$ denotes the canonical $\Spinc$ structure on $Y$. Let $n$ be a prime power and $\zeta_n$ be the primitive $n$th root of unity.  For $\Sigma_r(K)$, the $r$-fold cyclic cover of $S^3$ branched along a knot $K$, and a character of prime power order $\chi\colon H_1(\Sigma_r(K))\to \Q/\Z$, we denote by $\tau(K,\chi)\in L_0(\Q(\zeta_n))\otimes_\Z\Q$ the Casson-Gordon invariant associated to $\chi$ \cite{Casson-Gordon:1986-1}. 

Let
\[
B\ell_\Z\colon H_1(M(K);\Z[t^{\pm 1}])\times H_1(M(K);\Z[t^{\pm 1}]) \to S^{-1}\Z[t^{\pm 1}]/\Z[t^{\pm 1}]
\] 
be the Blanchfield form where $S:=\{f(t)\in \Z[t^{\pm1 }]\,\mid\, f(1)=1\}$. For a $\Z[t^{\pm 1}]$-module $M$ and $r\ge 1$, the map $\pi^r\colon M\to M/(t^r-1)M$ is the quotient map.

\begin{theorem}\label{theorem:metabelian-obstruction}
	Let $K$ and $J$ be knots with coprime Alexander polynomials. Suppose $K\# J$ is slice. Then, there exists a $\Z[t^{\pm 1}]$-submodule $P_K$ of $H_1(M(K);\Z[t^{\pm 1}])$ which is a metabolizer with respect to the Blanchfield form $B\ell_\Z$ for $K$ such that all of the following hold\textup{:} 
	\begin{enumerate}
		\item $P_K\otimes_\Z\Q$ is a metabolizer with respect to the rational Blanchfield form $B\ell_\Q$ for $K$ and $\rho(M(K),\phi_x)=0$ for all $x\in P_K\otimes_\Z\Q$. 
		\item There exists a set of finitely many primes $S=\{p_1,p_2,\ldots, p_m\}$ such that if $r=p^k$ is a prime power where $p\notin S$, then $\pi^r(P_K)$ is a metabolizer with respect to the linking form $H_1(\Sigma_r(K))\times H_1(\Sigma_r(K))\to \Q/\Z$ and the following hold\textup{:}
		\begin{enumerate}
			\item $\bar{d}(\Sigma_r(K),\mathfrak{s_0}+\hat{c})=0$ for all $c\in \pi^r(P_K)$, and
			\item for every character of prime power order $\chi\colon H_1(\Sigma_r(K))\to \Q/\Z$ which vanishes on $\pi^r(P_K)$, the Casson-Gordon invariant $\tau(K,\chi)$ is constant. Moreover, if $K$ is algebraically slice, then $\tau(K,\chi)$ vanishes.
		\end{enumerate} 
	\end{enumerate}
\end{theorem}

The proof of Theorem~\ref{theorem:metabelian-obstruction} is postponed to the end of Appendix. The set of primes $S$ in Theorem~\ref{theorem:metabelian-obstruction} can be taken as the empty set if $K\# J$ bounds a slice disk $D$ in the 4-ball such that, letting $W$ denote the exterior of the slice disk $D$, $H_1(W;\Z[t^{\pm 1}])$ has no $\Z$-torsion elements. In particular, if $K\# J$ is a ribbon knot, then $S$ is the empty set. 

\begin{proposition}\label{proposition:example}Let $R$ be the knot $9_{46}$, which is $R_1$ in Figure~\ref{figure:seed-knot}. Then, the knot $R(D_+(T_{2,3},0)\# -T_{2,3},D_+(T_{2,3},0))$ is topologically slice and not concordant to any knot with Alexander polynomial coprime to $\lambda_1=(t-2)(2t-1)$.
\end{proposition}
\begin{proof}
	Let $J:= D_+(T_{2,3},0)\#-T_{2,3}$ and $D:=D_+(T_{2,3},0)$. Let $K:=R(J,D)$. The knot $K$ is the satellite knot of winding number 0 whose pattern knot is $R(J,U)$ and companion knot is $D$. Therefore, since $D$ is topologically slice, $K$ is topologically slice. 
	
	Suppose $K\# L$ is slice for a knot $L$ whose Alexander polynomial is coprime to $\lambda_1$. Then by Theorem~\ref{theorem:metabelian-obstruction}, there exist a $\Z[t^{\pm 1}]$-submodule $P_K$ of $H_1(M(K);\Z[t^{\pm 1}])$, which is a metabolizer with respect to the Blanchfield form $B\ell_\Z$, and a set of finitely many primes $S$ which satisfies the conditions (1) and (2) in Theorem~\ref{theorem:metabelian-obstruction}. 
	
For brevity, let $P:=P_K$. Let $\alpha_J$ and $\alpha_D$ be the curves depicted in Figure~\ref{figure:seed-knot}. One can easily compute that 
\[
H_1(R(J,D);\Z[t^{\pm 1}])\cong H_1(R;\Z[t^{\pm 1}])\cong \Z[t^{\pm 1}]/\langle t-2\rangle \oplus \Z[t^{\pm 1}]/\langle 2t-1\rangle
\]
where the summands are $\langle \alpha_J \rangle$ and $\langle \alpha_D\rangle$, which are the submodules generated by the curves $\alpha_J$ and $\alpha_D$, respectively. One can also see that $R(J,D)$ and $R$ have isomorphic Blanchfield forms, and $\langle \alpha_J\rangle$ and $\langle \alpha_D\rangle$ are the only metabolizers with respect the Blanchfield form on $H_1(M(K);\Z[t^{\pm 1}])$. Therefore the submodule $P$ is either $\langle \alpha_J\rangle$ or $\langle \alpha_D\rangle$. 
	
	Suppose $P=\langle \alpha_J\rangle$. In this case we use Theorem~\ref{theorem:metabelian-obstruction}~(2)~(a); for an odd prime $r\notin S$, $\bar{d}(\Sigma_r(K),\mathfrak{s}_0+\hat{c})=0$ for all $c\in \pi^r(P)$. The subgroup $\pi^r(P)$ of $H_1(\Sigma_r(K))\cong \Z_{2^r-1}\oplus \Z_{2^r-1}$ is generated by a lift of $\alpha_J$ in $\Sigma_r(K)$, say $x$, and hence we have $\bar{d}(\Sigma_r(K),\mathfrak{s}_0+k\cdot\hat{x})=0$ for all $k\in \Z$. Recall from Remark~\ref{remark:nu+-equivalence} that $J$ is $\nu^+$-equivalent to the unknot. Since $\Sigma_r(K)$ is a $\Z_2$-homology sphere, by Theorem~\ref{theorem:d-invariant-practical}
	\[ 
	\bar{d}(\Sigma_r(K),\mathfrak{s}_0+k\cdot\hat{x})=\bar{d}(\Sigma_r(R(U,D)),\mathfrak{s}_0+k\cdot\hat{x})
	\] 
	for all $k\in \Z$. Since $R(U,D)$ is slice, $d(\Sigma_r(R(U,D)), \mathfrak{s}_0)=0$, and hence
	\[
	d(\Sigma_r(R(U,D)),\mathfrak{s}_0+k\cdot\hat{x})=\bar{d}(\Sigma_r(R(U,D)),\mathfrak{s}_0+k\cdot{x})=0
	\]
	for all $k\in \Z$. But by \cite[Theorem~5.4]{Cha-Kim:2017-1},
	\[
	d(\Sigma_r(R(U,D)),\mathfrak{s}_0+2^{r-1}\cdot\hat{x})\le -\frac32,
	\] 
	which is a contradiction.

	Suppose $P=\langle \alpha_D\rangle$. In this case, we use Theorem~\ref{theorem:metabelian-obstruction}~(1), which implies that $\rho(M(K),\phi_x)=0$ for all $x\in P\otimes_\Z \Q$. Choose a nonzero $x\in P\otimes_\Z\Q$, for instance, $x=\alpha_D\otimes 1$. Let $W$ be the exterior of the slice disk for $R(U,D)$ which is obtained by cutting the band dual to the curve $\alpha_D$ in Figure~\ref{figure:seed-knot}. Then $\partial W=R(U,D)$. Let $W'$ be a 4-manifold with $\partial W'=M(J)$ such that the inclusion induces an isomorphism $H_1(M(J))\xrightarrow{\cong} H_1(W')$. The existence of $W'$ follows from that the bordism group $\Omega_3(S^1)=0$. We can also arrange that $\pi_1(W')\cong \Z$  by doing surgeries along the curves generating $\pi_1(W')^{(1)}$. Note that $M(J)=(S^3\sm \nu(J))\cup (S^1\times D^2)$ where $\nu(J)$ is the open tubular neighborhood of $J$ in $S^3$. Let $X$ be the 4-manifold obtained from $W$ and $W'$ by identifying the tubular neighborhood of $\alpha_J$ in $R(U,D)=\partial W$ and $S^1\times D^2\subset M(J)=\partial W'$. Then $\partial X=M(K)$ and 
\[
\Ker\{H_1(M(K);\Q[t^{\pm 1}])\}\to H_1(X;\Q[t^{\pm 1}]) = P\otimes_\Z\Q. 
\]
Since $\Q[t^{\pm 1}]$ is a PID, $x\in P\otimes_\Z\Q$, and $P\otimes_\Z\Q$ is a metabolizer with respect to the rational Blanchfield form on $H_1(M(K);\Q[t^{\pm 1}])$, by \cite[Theorem~3.6]{Cochran-Orr-Teichner:2003-1} the homomorphism $\phi_x$ extends to $\pi_1(X)$.
	
	Therefore, $\rho(M(K),\phi_x)=\osigmat_\Gamma(W) + \osigmat_\Gamma(W')$. Since $W$ is a slice disk exterior, $\osigmat_\Gamma(W)=0$ by \cite[Theorem~4.2]{Cochran-Orr-Teichner:2003-1}. Since $\partial W'= M(J)$, we have $\osigmat_\Gamma(W')=\rho(M(J),\phi_x)$ where by abuse of notation $\phi_x$ also denotes the restriction of $\phi_x$ to $\pi_1(M(J))$. Since $x$ is a nonzero element in $P=\langle \alpha_D \rangle$, we have $B\ell_\Q(x,\alpha_J)\ne 0$ for the rational Blanchfield form $B\ell_\Q$ on $H_1(M(K),\Q[t^{\pm 1}])$. Also note that $\epsilon(\alpha_J)=0$ where $\epsilon \colon \pi_1(M(K))\to \Z$ is the abelianization. Since $\phi_x(\alpha_J)=(B\ell_\Q(x,\alpha_J), \epsilon(\alpha_J))$, the subgroup of $\Gamma$ generated by $\phi_x(\alpha_J)$ is isomorphic to $\Z$. By our construction of $X$, the meridian, say $\mu$, of $J$ is identified with the curve $\alpha_J$ in $X$, and hence $\phi_x(\mu)$ is not trivial in $\Gamma$. Since $\phi_x$ on $\pi_1(M(J))$ factors through $\pi_1(W')\cong \Z$, the image of $\pi_1(M(J))$ in $\Gamma$ is also isomorphic to $\Z$. Now by \cite[Proposition~5.13]{Cochran-Orr-Teichner:2003-1} and \cite[Lemma~5.3]{Cochran-Orr-Teichner:2004-1}, $\osigmat_\Gamma(W')=\rho(M(J),\phi_x)=\int_{S^1}\sigma_J(\omega)\,\,d\omega=\frac23$. Therefore $\rho(M(K),\phi_x)=0+\frac23\ne 0$, which is a contradiction.
\end{proof}

We finish Appendix with the proof of Theorem~\ref{theorem:metabelian-obstruction}.
\begin{proof}[{Proof of Theorem~\ref{theorem:metabelian-obstruction}}]
	Let $L:=K\# J$ and $W$ be the exterior of a slice disk for $L$ in the 4-ball. Let $FH_1(W;\Z[t^{\pm 1}]):=H_1(W;\Z[t^{\pm 1}])/T$ where $T$ is the $\Z$-torsion submodule of $H_1(W;\Z[t^{\pm 1}])$. Let 
	\[
	P_L:=\Ker\{i\colon H_1(M(L);\Z[t^{\pm 1}])\xrightarrow{i_\Z} H_1(W;\Z[t^{\pm 1}]) \xrightarrow{\pi} FH_1(W;\Z[t^{\pm 1}])\}
	\]
	where $i_\Z$ is the inclusion-induced homomorphism and $\pi$ is the quotient map. It is well known that $P_L$ is a metabolizer with respect to the the Blanchfield form $B\ell_\Z$ (see \cite[Proposition~2.7]{Friedl:2003-6}). 
	
	Note that $H_1(M(L);\Z[t^{\pm 1}])\cong H_1(M(K);\Z[t^{\pm 1}])\oplus H_1(M(J);\Z[t^{\pm 1}])$. Now let 
	\[
	P_K:=P_L\cap H_1(M(K);\Z[t^{\pm 1}])=\{x\in H_1(M(K);\Z[t^{\pm 1}])\,\mid\, (x,0)\in P_L\}.
	\]
	Similarly, we let $P_J:=P_L\cap H_1(M(J);\Z[t^{\pm 1}])$. We will show that $P_K$ satisfies the desired properties. We need a lemma:
	\begin{lemma}\label{lemma:splitting-metabollizers}
		$P_L\cong P_K\oplus P_J$ as $\Z[t^{\pm 1}]$-modules and $P_K$ and $P_J$ are metabolizers for the Blanchfield forms on $H_1(M(K);\Z[t^{\pm 1}])$ and $H_1(M(J);\Z[t^{\pm 1}])$, respectively.
	\end{lemma}
	\begin{proof}
		It is obvious that $P_K\oplus P_J\subset P_L$. Let $(x,y)\in P_L$ where $x\in H_1(M(K);\Z[t^{\pm 1}])$ and $y\in H_1(M(J);\Z[t^{\pm 1}])$. We will show that $x\in P_K$ and $y\in P_J$, which will imply that $P_L\subset P_K\oplus P_J$.

		Since $\Q[t^{\pm 1}]$ is a PID and $\Delta_K$ and $\Delta_J$ are coprime, there exist $\bar{f}(t)$ and $\bar{g}(t)$ in $\Q[t^{\pm 1}]$ such that $\bar{f}(t)\Delta_K+\bar{g}(t)\Delta_J=1$. Then, there exists some nonzero integer $c$ such that letting $f(t):=c\bar{f}(t)$ and $g(t):=c\bar{g}(t)$, we have $f(t)\Delta_K+g(t)\Delta_J=c$ and $f(t),\, g(t)\in \Z[t^{\pm 1}]$. Recall that $(x,y)\in P_L$. Since $\Delta_K$ and $\Delta_J$ annihilate $H_1(M(K);\Z[t^{\pm 1}])$ and $H_1(M(J);\Z[t^{\pm 1}])$, respectively, in $H_1(M(K);\Z[t^{\pm 1}])$ we have 
		\[
		cx=(f(t)\Delta_K+g(t)\Delta_J)x=(g(t)\Delta_J)x.
		\]
		Also, in $H_1(M(J);\Z[t^{\pm 1}])$ we have $(g(t)\Delta_J)y=0$. Therefore, in $H_1(M(L);\Z[t^{\pm 1}])$ we have $g(t)\Delta_J(x,y)=(cx,0)$, and hence $(cx,0)\in P_L$. 
		
		Consider the following commutative diagram.
		\[
		\begin{tikzcd}
		H_1(M(L);\Z[t^{\pm 1}]) \ar[d, "\otimes 1"]\ar[r, "i"] &
		FH_1(W;\Z[t^{\pm 1}])\ar[d, "\otimes 1"] 
		\\
		H_1(M(L);\Z[t^{\pm 1}])\otimes_\Z\Q \ar[r, "i\otimes \textrm{id}"]&
		FH_1(W;\Z[t^{\pm 1}])\otimes_\Z\Q 
		%      \\
		%      H_1(M(L);\Q[t^{\pm 1}]) \ar[r, "i_\Q"] &
		%      H_1(W;\Q[t^{\pm 1}])
		\end{tikzcd}
		\]
		Since $(cx,0)\in P_L=\Ker(i)$, it follows that $(i\otimes \textrm{id})((cx)\otimes 1, 0)=0$. Therefore 
		\[
		(i\otimes \textrm{id})(x\otimes 1, 0)=\frac1c \cdot (i\otimes \textrm{id})((cx)\otimes 1,0)=0,
		\] 
		and hence $(x\otimes 1, 0)\in \Ker(i\otimes \textrm{id})$. Since the vertical maps of the above diagram are injective, we have $(x,0)\in \Ker(i)=P_L$, and hence $x\in P_K$. Similarly, one can show that $y\in P_J$, and it follows that $P_L\subset P_K\oplus P_J$ and hence $P_L\cong P_K\oplus P_J$. The modules $P_K$ and $P_L$ are metabolizers due to \cite[Lemma~3.1]{Kim-Kim:2018-1} and its proof. 
	\end{proof} 
	
	One can see that $\rho(M(L), \phi_z)=0$ for all $z\in P_L\otimes_\Z\Q$ since 
	\[
	P_L\otimes_\Z\Q=\Ker\{i_\Q\colon H_1(M(L);\Q[t^{\pm 1}])\to H_1(W;\Q[t^{\pm 1}])\}
	\]
	where $W$ is the exterior of a slice disk (see Theorems~3.6 and 4.2 in \cite{Cochran-Orr-Teichner:2003-1}). By Lemma~\ref{lemma:splitting-metabollizers}, $P_L\otimes_\Z\Q\cong (P_K\otimes_\Z\Q)\oplus (P_J\otimes_\Z\Q)$, and $P_L\otimes_\Z\Q$,  $P_K\otimes_\Z\Q$, and $P_J\otimes_\Z\Q$ are metabolizers for the rational Blanchfield forms. Now by \cite[Theorem~3.1]{Kim-Kim:2008-1} and its proof, $\rho(M(K), \phi_x)=0$ for all $x\in P_K\otimes_\Z\Q$. This shows the property~(1). 
	
	We show the property~(2). Recall that $FH_1(W;\Z[t^{\pm 1}])=H_1(W;\Z[t^{\pm 1}])/T$ where $T$ is the $\Z$-torsion submodule. Since multiplication by $t-1$ induces an automorphism on $H_1(W;\Z[t^{\pm 1}])$ (see \cite{Milnor:1968-1}), by \cite[Lemma~(3.1)]{Levine:1977-1}, the subgroup $T$ is finite. Therefore, there are only finitely many primes $q_1, q_2, \ldots, q_\ell$ which divide the order of $T$. Then, by \cite[Lemma~2.3]{Bao:2015-1} there exists a set of finitely may primes $S=\{p_1, p_2,\ldots, p_m\}$ such that if $p$ is a prime not in $S$ and $r=p^k$ for some $k\ge 1$, then any $q_i$, which divides the order of $T$, does not divide the order of $H_1(\Sigma_r(K))$.
	
	Let $D$ be the slice disk for $L$ whose exterior is $W$. Denote the $r$-fold cyclic cover of the 4-ball branched along $D$ by $\Sigma_r(D)$. Then, it follows from \cite{Milnor:1968-1} that  
	\[
	\pi^r(H_1(M(L);\Z[t^{\pm 1}]))\cong H_1(\Sigma_r(L)) \textrm{ and }\pi^r(H_1(W;\Z[t^{\pm 1}]))\cong H_1(\Sigma_r(D)).
	\]
	
	Now suppose $r=p^k$, a prime power, where $p\notin S$. We have the following commutative diagram where  the homomorphisms $i_\Z$ and $j$ are induced from inclusions and $\pi$ is the quotient map.
	\[
	\begin{tikzcd}
	H_1(M(L);\Z[t^{\pm 1}]) \ar[d, "\pi^r"]\ar[r, "i_\Z"] &
	H_1(W;\Z[t^{\pm 1}])\ar[d, "\pi^r"] \ar[r, "\pi"] &
	FH_1(W;\Z[t^{\pm 1}])
	\\
	H_1(\Sigma_r(L))\ar[r, "j"] &
	H_1(\Sigma_r(D)) &
	\end{tikzcd}
	\]
	
	Restricting the right vertical map $\pi^r$ to $\Im(i_\Z)$, by our choice of $S$ and $r$, we obtain the following commutative diagram
	\[
	\begin{tikzcd}
	H_1(M(L);\Z[t^{\pm 1}]) \ar[d, "\pi^r"]\ar[r, "i_\Z"] &
	\Im(i_\Z) \ar[d, "\pi^r"] \ar[r, "\pi"] &
	\Im(\pi\circ i_\Z) \ar[dl, "\pi'"]
	\\
	H_1(\Sigma_r(L))\ar[r, "j"] &
	H_1(\Sigma_r(D)) &
	\end{tikzcd}
	\]
	where $\pi'$ is the quotient map. Now $\pi^r(P_L)\subset \Ker(j)$ since $P_L=\Ker(\pi\circ i_\Z)$. It is well known that $\Ker(j)$ is a metabolizer with respect to the linking form on $H_1(\Sigma_r(L))$. Also, since $P_L$ is a metabolizer for the Blanchfield form on $H_1(M(L);\Z[t^{\pm 1}])$, the subgroup $\pi^r(P_L)$ of $H_1(\Sigma_r(L))$ is a metabolizer for the linking form (see \cite[Proposition~2.18]{Friedl:2003-6}). Since $\pi^r(P_L)\subset\Ker(j)$ and $|\pi^r(P_L)|^2=|\Ker(j)|^2=|H_1(\Sigma_r(L))|$, it follows that $\pi^r(P_L)=\Ker(j)$. 
	
	Let $c\in \pi^r(P_K)$. We follow the arguments in \cite{Bao:2015-1} to show the property~(2)(a). Note that $\pi^r(P_L)=\pi^r(P_K)\oplus \pi^r(P_J)$, and hence $(c,0)\in \pi^r(P_L)$. Since $\pi^r(P_L)=\Ker(j)$ where $\Sigma_r(D)$ is the cyclic cover of the 4-ball branched along the slice disk $D$, it follows that $d(\Sigma_r(L), \mathfrak{s_0}+\hat{(c,0)})=0$. By the additivity of the $d$-invariant, we have 
	\[
	0=d(\Sigma_r(L), \mathfrak{s}_0+\hat{(c,0)})=d(\Sigma_r(K),\mathfrak{s}_0+\hat{c})+d(\Sigma_r(J),\mathfrak{s}_0),
	\]
	which holds for any choice of $c\in \pi^r(P_K)$. Therefore, for each $c\in \pi^r(P_K)$,
	\begin{align*}
		\bar{d}(\Sigma_r(K), \mathfrak{s}_0+\hat{c})&=d(\Sigma_r(K),\mathfrak{s}_0+\hat{c})-d(\Sigma_r(K),\mathfrak{s}_0)
		\\
		&=-d(\Sigma_r(J),\mathfrak{s}_0)-(-d(\Sigma_r(J),\mathfrak{s}_0)
		\\
		&=0,
	\end{align*}
	which shows the property~(2)(a). 
	
	To show the property~(2)(b), we follow the arguments in \cite{Kim:2005-2}. Let $\chi\colon H_1(\Sigma_r(K))\to \Q/\Z$ be a character of prime power order which vanishes on $\pi^r(P_K)$. Then the character 
	\[
	\chi\oplus 0\colon H_1(\Sigma_r(L))\cong H_1(\Sigma_r(K))\oplus H_1(\Sigma_r(J))\to \Q/\Z
	\]
	vanishes on $\pi^r(P_K)\oplus \pi^r(P_J) =\pi^r(P_L)$. Therefore, since $\pi^r(P_L)=\Ker(j)$, it follows that $\tau(L, \chi\oplus 0)=0$. Now by the additivity of the Casson-Gordon invariant, we have $0=\tau(L, \chi\oplus 0)=\tau(K,\chi)+\tau(J,0)$. Therefore, $\tau(K,\chi)=-\tau(J,0)$, a constant. If $K$ is algebraically slice, since $K\# J$ is slice by the hypothesis, $J$ is also algebraically slice, and it follows that $\tau(J,0)=0$. 
\end{proof}

\providecommand{\bysame}{\leavevmode\hbox to3em{\hrulefill}\thinspace}
\providecommand{\MR}{\relax\ifhmode\unskip\space\fi MR }
\providecommand{\MRhref}[2]{%
  \href{http://www.ams.org/mathscinet-getitem?mr=#1}{#2}
}
\providecommand{\href}[2]{#2}

\end{document}